\documentclass[12pt,a4paper,reqno]{amsart}            
\usepackage[utf8]{inputenc}
\usepackage{a4wide}

\title{Calderón-Zygmund theory on some Lie groups of exponential growth}

\author[F. \ De Mari]{Filippo De Mari}
\address[Filippo De Mari]{Dip. di Matematica, Dipartimento di Eccellenza 2023-2027, and MaLGa center, Università di Genova, Via Dodecaneso 35, 16146 Genova, Italy}
\email{\href{mailto: demari@dima.unige.it}{\texttt{demari@dima.unige.it}}}

\author[M.\ Levi]{Matteo Levi}
\address[Matteo Levi]{Fac. de Ciències, Universitat Autònoma de Barcelona, 08193 Bellaterra, Spain, and Dept.\ Matem\`atica i Inform\`atica,
	Universitat  de Barcelona,
	Gran Via 585, 08007 Bar\-ce\-lo\-na, Spain }
\email{\href{mailto: matteo.levi@ub.edu}{\texttt{matteo.levi@ub.edu}}}

\author[M. \ Monti]{Matteo Monti} \address[Matteo Monti]{Dip. di Scienze Matematiche ``Giuseppe Luigi Lagrange''  \\
	Politecnico di Torino\\ corso Duca degli Abruzzi 24\\ 10129 Torino\\ Italy}
\email{\href{mailto: matteo.monti@polito.it}{\texttt{matteo.monti@polito.it}}}

\author[M. Vallarino]{Maria Vallarino}
\address[Maria Vallarino]{Dip. di Scienze Matematiche ``Giuseppe Luigi Lagrange''  \\
	Politecnico di Torino\\ corso Duca degli Abruzzi 24\\ 10129 Torino\\ Italy}
\email{\href{mailto: maria.vallarino@polito.it}{\texttt{maria.vallarino@polito.it}}}

\keywords{Calder\'on--Zygmund theory; nondoubling spaces; exponential growth groups; Hardy-Littlewood maximal function}

\thanks{
	The authors are members of the Gruppo Nazionale per l'Analisi Matematica, la Probabilit\`a e le loro Applicazioni (GNAMPA) of the
	Istituto Nazionale di Alta Matematica (INdAM). This work is partially supported by the project "Harmonic analysis on continuous and discrete structures" funded by Compagnia di San Paolo (Cup E13C21000270007). M. Levi has been partially supported by the Spanish Ministerio de Ciencia e Innovaci\'on
	(projects PID2021-123405NB-I00 and PID2021-123151NB-I00), by the Departament de Recerca i Universitats, project 021 SGR 00087, by the 
	Generalitat de Catalunya (grant
	2021 SGR 00071) and by the INdAM GNAMPA Project CUP E53C22001930001.
}

\usepackage{amsthm,amssymb,amsmath,amsopn,amsfonts,pb-diagram,cases,accents,xcolor,array,hyperref,mathtools,enumitem,bm,mathrsfs}

\usepackage{authblk}

\usepackage{comment}
\usepackage{algpseudocode}
\usepackage{multicol}

\newcolumntype{P}[1]{>{\centering\arraybackslash}p{#1}}
\newcolumntype{M}[1]{>{\centering\arraybackslash}m{#1}}

\newcommand{\R}{\mathbb{R}}
\newcommand{\HH}{\mathbb{H}}
\newcommand{\ga}{\mathfrak{a}}

\newcommand{\Sptwor}{\mathrm{Sp}(2,\R)}
\newcommand{\Span}{\mathrm{span}}
\newcommand{\de}{{\rm d}}
\newcommand{\arccosh}{{\rm arccosh}}

\DeclareMathOperator{\diam}{diam}

\theoremstyle{plain}

\newtheorem{thm}{Theorem}[section]

\newtheorem{lem}[thm]{Lemma}
\newtheorem{prop}[thm]{Proposition}

\theoremstyle{definition}
\newtheorem{defn}[thm]{Definition} 
\newtheorem{oss}[thm]{Remark}

\theoremstyle{remark}

\newcommand{\RR}{\mathbb{R}}

\newcommand{\zz}{\mathbb{Z}}

\newcommand{\CZ}{Calder\'on--Zygmund }
\newcommand{\CaC}{Carnot--Cara\-th\'eodory }

\newcommand{\ndi}{\mathrm{d}}  
\newcommand{\di}{\,\ndi}    

\newcommand{\step}{S} 
\newcommand{\vfG}{X}
\newcommand{\vfA}{\breve{X}_0}
\newcommand{\vfN}{\breve{X}}

\newcommand{\diag}{{\rm diag}}

\newcommand{\lie}{\mathfrak}
\newcommand{\dimhom}{M}

\newcommand{\dil}[2]{{#1}^{(#2)}}

\usepackage{mathtools}

\begin{document}

	\maketitle

	\begin{abstract}
		Let $G = N \rtimes A$, where $N$ is a stratified Lie group and $A= \mathbb R_+$ acts on $N$ via automorphic dilations. We prove that the group $G$ has the \CZ property, in the sense of Hebisch and Steger, with respect to a family of flow measures and metrics. This generalizes in various directions previous works by Hebisch and Steger and Martini, Ottazzi and Vallarino, and provides a new approach in the development of \CZ theory in Lie groups of exponential growth. We also prove a weak type $(1,1)$ estimate for the Hardy--Littlewood maximal operator naturally arising in this setting. 
	\end{abstract}

	\section{Introduction}
	
	In the past century, the classical \CZ theory has been developed in the Euclidean setting and, more generally, on spaces of homogeneous type, see among others~\cite{CW,gra,S}. In the following years, many efforts have been made to generalize such theory in various nondoubling settings, both of polynomial and exponential growth (see for example~\cite{CMM,Conde,CondeParcet,dlHH,HS,MOV,arthur,NTV,T,To,Ve,Va}). 
	
	In this paper, we are specifically interested in the approach of one of the contributions listed above. Namely, the seminal paper \cite{HS} of some 20 years ago by Hebisch and Steger, in which they introduced an abstract \CZ theory based on the following definition.

	\begin{defn}\label{def CZP}
		A metric measure space $(X,d,\mu)$, with $\mu(X)=\infty$, has the \CZ property (CZP) if there exists $C_0\geq 1 $ such that, for every $f\in L^1(\mu)$ and $\alpha>0$, there exist a countable family of sets $\mathcal{E}(f,\alpha)=\{E_j\}$, positive numbers $r_j$, and points $x_j\in X$ for which $f=g+\sum_j b_j$, in such a way that, for every $j\in \mathbb{N}$,
		\begin{itemize}
			\item[(a)] $|g| \le C_0 \alpha$ $\mu$-almost everywhere;
			\item[(b)]  $b_j=0$ on $X\setminus E_j$;
			\item[(c)]  $\displaystyle\sum_j\|b_j\|_1 \le C_0\|f\|_1$ and $\displaystyle\int_{E_j} \ b_j \ d\mu = 0$;  
			\item[(d)] $E_j\subset B(x_j,C_0r_j)$;
			\item[(e)] $\displaystyle\sum_j \mu(E_j^*) \le \frac{C_0}{\alpha}\|f\|_1$, where $E_j^*=\{x: \ d(x,E_j)<r_j\}$.
		\end{itemize}
		In such case, we let $\mathcal{E}=\{E\in \mathcal{E}(f,\alpha): f\in L^1(\mu), \ \alpha>0\}$, and we say that $(X,d,\mu)$ has the CZP with respect to the family $\mathcal E$, and that $\mathcal{E}$ is a CZ family for $(X,d,\mu)$.
	\end{defn}
	
	Observe that properties (a), (b) and (c) in Definition \ref{def CZP} only concern $(X,\mu)$ as a measure space, and do not depend in any way from the choice of a metric $d$ on $X$. We will say that $L^1(\mu)$ admits a CZ decomposition with respect to the family $\mathcal E$ if (a), (b) and (c) hold true.
	
	In \cite{HS} the authors provided evidence that spaces enjoying the CZP constitute a fertile environment to develop harmonic analysis (in particular singular integrals theory) which goes beyond the comfort zone of the spaces of homogeneous type. Indeed, while all spaces of homogeneous type have the CZP, the class of spaces with the CZP is strictly larger, and it even includes some natural and well studied spaces of exponential growth. In the discrete setting, a first example of such a class is provided by homogeneous trees with the natural distance and the canonical flow measure~\cite{HS}. It was later shown in \cite{LSTV} that the CZP actually extends to any tree (non necessarily homogeneous) with any locally doubling flow measure (not necessarily the canonical one). Also in the continuous setting, on which we focus in this paper, there exist nontrivial examples of spaces of exponential growth enjoying the CZP. Consider the group $G = N \rtimes A$, where $N$ is a stratified Lie group and $A= \mathbb R_+$ acts on $N$ via automorphic dilations. Let $d_G$ be a suitably chosen \CaC metric on $G$ and $\rho$ a right Haar measure on $G$. Then, $(G,d_G,\rho)$ has the CZP. This result was first proved in \cite[Lemma 5.1]{HS} for the case $N=\mathbb{R}^n$ (i.e, when $G$ is a so called $ax+b$ group), and then extended to the general case of arbitrary stratified Lie group $N$ in \cite[Theorem 3.20]{MOV}.
	
	The aim of this paper is to enrich further the fauna of noncompact Lie groups of exponential growth treatable in the context of the abstract \CZ theory described above.
	Our setting is the following. We consider the same groups $G$ as in \cite{MOV} and a class $\mathcal Z$ of left-invariant vector fields having nonvanishing  vertical component.
	Given a vector field $Z\in \mathcal Z$, we say that a measure $\mu$ is a $Z$-flow if it is absolutely continuous with respect to $\rho$ and its Radon--Nykodim derivative $\varphi$ is right--invariant with respect to the multiplication by $\exp(tZ)$, $t\in\mathbb R$. We
	introduce the class $\mathcal{F}_Z$ of measures $\mu$ on $G$ that are $Z$-flows and such that $(N,d_N,\mu_N)$ is doubling, where $\de\mu_N(n):=\varphi(n,1)\de n$ (see Section~\ref{sec: preliminaries}) and $d_N$ is a \CaC metric on $N$.
	We then construct an associated family $\mathcal{D}^Z$ of subsets of $G$, and we define a flow metric $d_Z$ (see Sections \ref{sec: admissible sets and dyadic partitions}, and \ref{sec: CZP vertical flow}, respectively, for their precise definitions). It is important to point out that the space $(G,d_Z,\mu)$ has exponential growth, and hence it is nondoubling, at least when $\varphi$ is bounded away from zero on $G$. Our main result is the following.
	\begin{thm}\label{main}
		For every vector field $Z\in\mathcal Z$ and any measure $\mu\in \mathcal{F}_Z$ the metric measure space $(G,d_Z,\mu)$ has the CZP with respect to the family $\mathcal{D}^Z$.
	\end{thm}
	
	Theorem~\ref{main} is strictly more general than the result of \cite{MOV} previously cited. Indeed, right Haar measures belong to $\mathcal{F}_Z$ for any $Z\in\mathcal Z$, so that in Theorem \ref{main} one can always choose, in particular, $\mu=\rho$, as in \cite{MOV}. Moreover, (see Section \ref{sec: CZP for dg}) when $Z$ is the vertical vector field $X_0$ (see Section \ref{sec: preliminaries}), then $d_Z=d_G$. That said, \cite[Theorem 3.20]{MOV} can be rephrased saying that $(G,d_Z,\mu)$ has the CZP with respect to the family $\mathcal{D}^Z$ if $Z=X_0$ and $\mu=\rho$.

	We point out that our proof of Theorem \ref{main} is new (even in the known case in which $d_Z=d_G$ and $\mu=\rho$), since we cannot exploit the left-invariance of the metric nor that of the measure.
	
	As mentioned before, having the CZP property is a key ingredient to develop a theory of singular integrals on a metric measure space. In particular, Theorem \ref{main} implies boundedness properties for a class of linear integral operators on $(G,d_Z,\mu)$ whose kernels satisfy a Hörmander type condition, see Theorem \ref{integral operators 1}.
	A first natural project in this direction would be the study of the boundedness properties of the Riesz transform associated to a flow Laplacian on $G$ (i.e., a Laplacian operator self-adjoint on $L^2(\mu)$). Such boundedness has already been studied on the $ax+b$ group  (in \cite{GaudrySjogren, HS, Martini, Sjogren}), on homogeneous trees (in \cite{HS} and \cite{LMSTV}), and on nonhomogeneous trees in~\cite{MSTV}. We are not addressing this or other applications here, leaving it for a possible follow-up work or for other interested mathematicians.

	We now briefly describe the structure of the paper. In Section \ref{sec: preliminaries} we recall the basic notions on stratified Lie groups and their rank one extensions, and we introduce the family of vector fields $\mathcal{Z}$ and the class of measures $\mathcal{F}_Z$, where $Z\in \mathcal{Z}$.
	
	In Section \ref{sec: admissible sets and dyadic partitions} we introduce the class of \textit{admissible cylinders} in $G$, which resembles the class of admissible sets first appeared in the $ax+b$ group in \cite{GiuliniSjogren}, which in turn inspired \cite{HS, MOV}. We then prove that 
	if $\mu\in \mathcal{F}_Z$, then $(G,\mu)$ admits a family $\mathcal{D}^Z$ of (admissible) dyadic sets (Theorem \ref{thm dyadic partitions}). As a consequence, we deduce that $L^1(\mu)$ admits a CZ decomposition with respect to $\mathcal{D}^Z$ (Theorem \ref{CZdec}).
	
	In Section \ref{sec: maximal}, we consider the problem of the boundedness of the Hardy--Littlewood maximal function associated to the admissible cylinders  introduced in Section \ref{sec: admissible sets and dyadic partitions}. By means of a covering lemma for admissible sets (Proposition \ref{vitali}), in Theorem \ref{maximal theorem} we are able to show that the maximal function is of weak type (1,1).
	
	In Section \ref{sec: CZP vertical flow}, we introduce a flow metric $d_Z$ and, by means of some geometric lemmas, we are able to show that (d) and (e) in Definition \ref{def CZP} are satisfied on $(G,d_Z,\mu)$ by the sets in $\mathcal{D}^Z$. This, together with Theorem \ref{CZdec}, completes the proof of our main result, Theorem \ref{main}.
	
	In Section \ref{sec: CZP for dg} we compare our result with those previously available in the literature. First we observe, as mentioned before, that when the vector field $Z$ is vertical then $d_Z=d_G$, and therefore the result of \cite{MOV} (and, a fortiori, that of \cite{HS}) can be improved to: $(G,d_G,\mu)$ has the CZP with respect to the family $\mathcal{D}^{Z}$ for \textit{any} $\mu\in\mathcal{F}_{Z}$ (and not only $\mu=\rho$) if $Z$ is the vertical vector field. Next, we prove that when $N=\mathbb R^n$ we can even say more. Indeed, we show that in this case $d_G$ is equivalent to $d_Z$ for \textit{any} $Z\in \mathcal{Z}$. Hence, we can improve the result of \cite{HS} further to: for $N=\mathbb R^n$, $(G,d_G,\mu)$ has the CZP with respect to the family $\mathcal{D}^Z$ for any $Z\in \mathcal{Z}$, and any $\mu\in\mathcal{F}_Z$. We dedicate the last part of the section, and of the paper, to investigate whether, always with $d_G$ as underlying metric, the same level of generality in the choice of the family of sets and of the measure can be attained also when $N$ is nonabelian.  We give negative answer to this question by providing a counterexample in the extended Heisenberg group $\mathbb{H}^1_e$, also known as the shearlet group~\cite{DahDeMDeV}. In particular, in Theorem \ref{thm: counterexample} we consider a particular vector field $Z$ and we show that $(\mathbb{H}^1_e,d_{\mathbb{H}^1_e},\rho)$, which is known to have the CZP with respect to the family $\mathcal D^{X_0}$ already from~\cite{MOV}, does not have the CZP with respect to $\mathcal{D}^Z$.

	Throughout the work, we write $f(x)\lesssim g(x)$ if there exist a uniform constants $C>0$, such that $f(x)\leq C g(x)$, for every $x$, and we write $f(x)\approx g(x)$ if it is both $f(x)\lesssim g(x)$ and $g(x)\lesssim f(x)$. Constants carrying a numerical subscript, such as $C_1, C_2, \dots$ are meant to maintain their value across the whole paper, while $C$ will be used (typically in proofs) for a generic constant whose value may change from line to line.

	\section{Preliminaries and notation}\label{sec: preliminaries}

	A Lie algebra $\mathfrak n$ is said to be stratified of step $\step\in\mathbb N$, $\step\geq 1$, if it admits a vector space decomposition
	\[\mathfrak n=\bigoplus_{j=1}^\step \mathfrak n_j,\quad\text{with}\quad [\mathfrak n_i,\mathfrak n_j]\subset \mathfrak n_{i+j}.\]
	Every $\mathfrak n_j$ is a layer and $\dimhom:=\sum_{j=1}^\step j\dim(\mathfrak n_j)$
	is called the homogeneous dimension of $\mathfrak n$. A stratified Lie algebra $\mathfrak{n}$ can be equipped with a derivation $\partial$ such that $\mathfrak n_j$ is the eigenspace of $\partial$ corresponding to the eigenvalue $j$.
	A stratified Lie group $N$ is a simply connected Lie group whose Lie algebra $\mathfrak n$ is stratified.

	Any stratified Lie algebra, and then group, is nilpotent, hence unimodular. The push-forward of the Lebesgue measure on $\lie{n}$ via $\exp_N\colon\mathfrak n\to N$ is a left and right Haar measure on $N$, which we fix and denote by $dn$. The formula $D_a = \exp_N((\log a) \partial)$ defines a family of automorphic dilations $(D_a)_{a\in A}$ on $N$. Hence the Lie group $A=(\R_+,\cdot)$ acts on $N$ via $D_a\colon N\to N$ and we can consider the corresponding semidirect product $G=N\rtimes A$, namely the product $N\times A$ endowed with the multiplication
	\begin{equation*}
		(n,a)(n',a'):=(nD_a(n'),aa'),\qquad n,n'\in N\text{, }\,\,a,a'\in A.
	\end{equation*}
	The neutral element of $G$ is $1_G=(1_N,1)$ and the inverse of $(n,a)\in G$ is $(n,a)^{-1}=(D_{1/a}(n^{-1}),1/a)$. The group $G$ is a solvable Lie group, and the Lie algebra $\lie{g}$ of $G$ is naturally identified with the semidirect product of Lie algebras $\lie{n} \rtimes \lie{a}$ (see \S 3.14-3.15 in~\cite{Varadarajan}), namely $\mathfrak g=\mathfrak n\oplus \mathfrak a$, with
	\[ [(X,Y),(X',Y')]_\mathfrak g:=([X,X']_\mathfrak n+\partial_Y X'-\partial_{Y'} X,0),\qquad X,X'\in\mathfrak n,\,Y,Y'\in\mathfrak a,\] 
	where $\partial_Y$ denotes the differential at $Y\in \mathfrak a$ of the map $a\mapsto D_a$, hence a derivation of $\mathfrak a$.
	A left and a right Haar measures $\lambda$ and $\rho$ on $G$ are given  by
	\begin{equation*}
		\di\lambda(n,a) = a^{-\dimhom -1} \di n \di a  \qquad \di\rho(n,a) =a^{-1} \di n \di a,    
	\end{equation*}
	respectively.
	In particular $G$ is not unimodular.

	We put $q_j=\dim\mathfrak n_j$, $1\leq j\leq \step$, and consider
	a basis $\{\vfN_{j,i}\colon 1\leq i\leq q_j\}$ of $\mathfrak n_j$. 
	We fix a scalar product on $\mathfrak n$ that makes 
	$\{\vfN_{j,i}\colon 1\leq j\leq \step,\,1\leq i\leq q_j\}$
	an orthonormal basis of $\mathfrak n$.
	Consequently, $\{\vfN_{1,1},\dots,\vfN_{1,q_1}\}$ is an orthonormal basis of $\mathfrak n_1$ and provides a subbundle $HN\subset TN$ that is called horizontal. We say that a curve $\gamma_N\colon [0,1]\to N$ of $N$ is horizontal if $\dot{\gamma}_N(t)\in HN$ for every $t\in (0,1)$. The \CaC distance $d_N(n,n')$ between two elements $n,n'\in N$ is given by the infimum of the lengths of the horizontal curves joining $n$ and $n'$. Since the horizontal distribution that makes $\vfN_{1,1},\dots,\vfN_{1,q_1}$ into an orthonormal basis is bracket-generating, the distance $d_N$ is finite and induces on $N$ the usual topology.
	Moreover the distance $d_N$ is left-invariant and homogeneous with respect to the automorphic dilations $D_a$, namely $d_N(D_a(n),D_a(n'))=a^\dimhom d_N(n,n')$, for every $n,n'\in N$ and $a\in A$.

	The vector fields $\vfN_{j,i}\in\mathfrak n$ introduced before can be lifted to left-invariant vector fields on $G$ by the formula
	\[\vfG_{j,i}|_{(n,a)} := a \vfN_{j,i}|_n \qquad \text{for}\quad j=1,\dots,\step,\quad i=1,\dots,q_j.\]
	Let $\vfA=a\frac{\de}{\de a}$ be the canonical basis on $\lie{a}$. We lift it to $G$ by \[\vfG_0|_{(n,a)} := \vfA|_a  .\]
	The system $\{\vfG_0,\vfG_{1,1},\dots,\vfG_{1,q_1}\}$ generates the Lie algebra $\lie{g}$ and defines a sub-Riemannian structure on $G$ with associated horizontal distribution $HG$, sub-Riemannian metric $g$ and left-invariant \CaC distance $d_G$.
	The following relation between the \CaC distances on $G$ and $N$ is proved in \cite[Proposition 2.7]{MOV}.
	
	\begin{prop}\label{prp:distance}
		For all $(n,a),(n'a')\in G$,
		\begin{equation}\label{CC distance}
			\cosh\Big(d_G((n,a),(n',a'))\Big)=\cosh\Big(\log\frac{a}{a'}\Big)+\frac{1}{2aa'}d_N(n,n')^2.
		\end{equation}
	\end{prop}
	Now we consider the scalar product $\langle\cdot,\cdot\rangle\colon\mathfrak g\times \mathfrak g\to \mathbb C$ that makes 
	\[\{X_0\}\cup\{\vfG_{j,i}\colon 1\leq i\leq q_j,\, 1\leq j\leq \step\}\]
	an orthonormal basis of $\mathfrak g$. We consider the vector fields with nonvanishing, hence normalized, vertical component, namely
	\[\mathcal Z:=\{Z\in\mathfrak g\colon \langle Z,X_0\rangle=1\}.\]
	Observe that for every $Z\in \mathcal Z$,
	\[Z-X_0\in\{X_0\}^\perp=\Span\{X_{j,i}\colon 1\leq i\leq q_j, 1\leq j\leq \step\}=\mathfrak n\oplus\{0\}\subset\mathfrak g.\] 
	Since
	for every $t\in\mathbb R$, $\exp(tX_0)=(1_N,e^t)$,
	there exists $n(t)\in N$ such that
	\[(n(t),1)=\exp(t(Z-X_0))=\exp(tZ)\exp(-tX_0),
	\] 
	namely,
	\[\exp{(tZ)}=(n(t),e^{t}),\quad t\in \mathbb{R}.\]

	\begin{defn}
		Given $Z\in \mathcal{Z}$, we say that a Borel measure $\mu$ on $G$ is a $Z$-\textit{flow measure} if it is absolutely continuous  with respect to the right Haar measure~$\rho$, and the Radon-Nikodym derivative $\varphi\colon N\times A\to [0,+\infty)$ is such that
		\begin{equation}\label{invar def}    \varphi(n,a)=\varphi((n,a)\exp({tZ})), \quad \text{for every } (n,a)\in G,\,\,\, t\in\RR.
		\end{equation}
		
	\end{defn}

	Clearly, if $\mu$ is a $Z$-flow measure, then
	\begin{equation}\label{invariance}
		\mu(E\exp(tZ))=\mu(E),\qquad\text{for every Borel set }E\subset G,\,\,t\in\R.
	\end{equation}
	
	We associate to $\mu$ a measure on $N$ given by
	\begin{equation*}
		\mu_N(F)=\int_F \varphi (n,1)dn, \quad \text{for every Borel set } F\subset N.
	\end{equation*}

	In this paper we will only consider $Z$-flow measures $\mu$ such that $(N,d_N,\mu_N)$ is a doubling metric measure space. Therefore we put 
	\[ \mathcal{F}_Z:=\{Z\text{-flow measures }\mu\colon (N,d_N,\mu_N)\text{ is doubling}\}.\]

	We recall that a measure $\nu$ on a metric space $X$ is doubling if for every $C>1$ there exists a constant $D(\nu,C)>1$ such that
	\begin{equation}\label{doubling}
		\nu(B(x,C r))\leq D(\nu,C) \nu(B(x,r)), \quad \text{for every } x\in X, r>0.
	\end{equation}

	Observe that the right Haar measure $\rho$ is in $\mathcal{F}_Z$ with respect to any vector field $Z\in\mathcal Z$, since $\varphi\equiv 1$ satisfies~\eqref{invar def}. Furthermore, in such case, $\rho_N$ is a Haar measure of $N$, which is doubling with respect to a \CaC metric on $N$.

	\begin{oss}
		Observe that, for any $Z\in\mathcal{Z}$, \textit{any} doubling measure $\mu_N$ on $N$ which is absolutely continuous with respect to $dn$, can be associated to some $Z-$flow in the sense of the above definition. Indeed, if $\psi:N\to [0,+\infty)$ is the density of $d\mu_N$ with respect to $dn$, then it is not difficult to see that the function $\varphi:G\to [0,+\infty)$ defined by $\varphi(n,a)=\psi(n n(\log a)^{-1})$ is such that $\varphi(n,1)=\psi(n)$ and satisfies \eqref{invar def}, so that the measure on $G$ having $\varphi$ as a density is in $\mathcal{F}_Z$ and $\mu_N$ is its associated measure on $N$. 
	\end{oss}

	\section{Admissible cylinders and dyadic partitions}\label{sec: admissible sets and dyadic partitions}
	
	In this section we first define a family of sets in $G$ which we call \textit{cylinders} and we discuss a number of useful properties they enjoy. Then we introduce the subfamily of the \textit{admissible cylinders}. Finally, we prove the existence of a family $\mathcal{D}^Z$ of dyadic partitions of $G$ made of admissible cylinders, which leads to a CZ  decomposition for functions in $L^1(\mu)$.

	\subsection{Cylinders}
	\begin{defn}\label{def cil}
		Let $E$ be any subset of $N$, $r>1$ and $a\in  A$. The cylinder $P_{r,E}(a)$ is defined by
		\begin{equation*}
			P_{r,E}(a)=\Big\{(n,1)\exp{(tZ)}: n\in E, \ t\in U_r(a)\Big\},
		\end{equation*}
		where
		\begin{equation}\label{Ur}
			U_r(a)=\Big(\log\Big(\frac{a}r\Big),\log(ar)\Big)\subset\R.
		\end{equation}
		We say that $E$ is the base set of $P_{r,E}(a)$.
	\end{defn}
	
	The next proposition collects some properties enjoyed by cylinders which will be useful in the sequel.
	
	\begin{prop}\label{prop-cyl}
		For any $Z\in \mathcal{Z}$, every $s\in \mathbb R$, $r, r_1,r_2>1$, $a,a_1,a_2\in A$, $E, E_1,E_2\subset N$, $m\in N$, the following hold:
		\begin{itemize}
			\item[(i)] $P_{r,E}(a)\exp{(sZ)}=P_{r,E}(ae^{ s})$;
			\item[(ii)] $\exp{(sZ)}P_{r,E}(a)=P_{r,\psi_s(E)}(a e^{ s})$,\quad\text{where}\quad 
			$\psi_s(m):=n(s)D_{e^s}(m)n(s)^{-1}$;
			\item[(iii)] $(m,1)P_{r,E}(a)=P_{r,mE}(a)$;
			\item[(iv)] two cylinders $P_i=P_{r_i,E_i}(a_i)$, $i=1,2$, intersect if and only if $E_1\cap E_2\neq\emptyset$ and $U_{r_1}(a_1)\cap U_{r_2}(a_2)\neq \emptyset$;
			\item[(v)] let $P_i=P_{r_i,E_i}(a_i)$, $i=1,2$. Then,
			\begin{equation*}
				P_1P_2\supset P_{r_1r_2,E_1\cdot \Psi_{r_1,a_1}(E_2)}(a_1a_2),\quad\text{where}\quad\Psi_{r,a}(E'):=\bigcap_{t\in U_{r}(a)}\psi_t(E');
			\end{equation*}
			\item[(vi)] if $\mu$ is a $Z$-flow measure, then    $$\mu(P_{r,E}(a))=2\mu_N (E)\log r .$$
		\end{itemize}
	\end{prop}
	\begin{proof}
		Property $(i)$ is immediate since
		\begin{equation*}
			\begin{split}
				P_{r,E}(a)\exp{(sZ)}&=\Big\{ (n,1)\exp{((t+s)Z)}: n\in E, t+s\in U_r(a)+s\Big\}
			\end{split}
		\end{equation*}
		and 
		\begin{equation}\label{traslaU}
			U_r(a)+s=U_r(ae^s).
		\end{equation}
		To prove $(ii)$, first observe that
		\begin{equation*}
			\begin{split}
				\exp(sZ)(n,1)&=(n(s)D_{e^s}n,e^s)=(\psi_s(n)n(s),e^s)=(\psi_s(n),1)\exp(sZ).
			\end{split}  
		\end{equation*}
		Clearly $n\in E$ if and only if  $\psi_s(n)\in\psi_s(E)$.
		Hence we have,  by~\eqref{traslaU}
		\begin{equation*}
			\begin{split}
				\exp{(sZ)}P_{r,E}(a)&=\Big\{\exp(sZ)(n,1)\exp{(tZ)}:  n\in E, \ t\in U_r(a)\Big\}\\
				&=\Big\{( \psi_s(n),1)\exp{((s+t)Z)}:  n\in E, \ s+t\in U_r(ae^s)\Big\}\\
				&=\Big\{( n',1)\exp{ ( t' Z)}:  n'\in\psi_s(E), \ t'\in U_r(ae^s)\Big\}\\
				&=P_{r, \psi_s(E)}(a e^{ s}).
			\end{split}
		\end{equation*}
		It is  straightforward to get $(iii)$:
		\begin{equation*}
			\begin{split}
				( m,1)P_{r,E}(a)&=\Big\{ ( m,1)( n,1):  n\in E,\ t\in U_r(a)\Big\}\\ 
				&=\Big\{ (mn,1):  mn\in  mE,\ t\in U_r(a)\Big\}=P_{r, mE}(a).
			\end{split}  
		\end{equation*}
		We now prove $(iv)$. Two cylinders $P_i=P_{r_i,E_i}(a_i)$, $i=1,2$, intersect if and only if there exists $n_i\in E_i$ and $t_i\in U_{r_i}(a_i)$ such that
		\begin{equation*}
			(n_1,1)=(n_2,1)\exp((t_2-t_1)Z)=(n_2n(t_2-t_1),e^{t_2-t_1}).
		\end{equation*}
		This is possible if and only if $t_1=t_2$ and $n_1=n_2$. Therefore, $P_1\cap P_2\neq\emptyset$ if and only if $E_1\cap E_2\neq\emptyset$ and $U_{r_1}(a_1)\cap U_{r_2}(a_2)\neq \emptyset$, which is  $(iv)$.
		
		We now turn to $(v)$. First observe that by means of $(ii)$ and $(iii)$ we have
		\begin{equation*}
			(n,1)\exp(tZ)P_2=(n,1)P_{r_2, \psi_t(E_2)}(a_2e^t)=P_{r_2,n \psi_t(E_2)}(a_2e^t).
		\end{equation*}
		Therefore, we can write
		\begin{equation*}
			\begin{split}
				P_1P_2&= \bigcup_{x\in P_1}xP_2= \bigcup_{n\in E_1}\bigcup_{t\in U_{r_1}(a_1)}(n,1)\exp(tZ)P_2\\
				&= \bigcup_{n\in E_1}\bigcup_{t\in U_{r_1}(a_1)}P_{r_2,n \psi_t(E_2)}(a_2e^t)= \bigcup_{t\in U_{r_1}(a_1)}P_{r_2,E_1 \psi_t(E_2)}(a_2e^t).
			\end{split}    
		\end{equation*}
		But for any $t\in U_{r_1}(a_1)$ we have $E_1\psi_t(E_2)\supset E_1\Psi_{r_1,a_1}(E_2)$, and hence
		\begin{equation*}
			\begin{split}
				P_1P_2&=\bigcup_{t\in U_{r_1}(a_1)}P_{r_2,E_1\psi_{t}(E_2)}(a_2e^{t})
				\\&=\bigcup_{t\in U_{r_1}(a_1)}\Big\{ (m,1)\exp(sZ): m\in E_1 \psi_{t}(E_2), \ s\in U_{r_2}( a_2e^{t})\Big\}\\
				&\supset \Big\{ (m,1)\exp(sZ): m\in  E_1\Psi_{r_1,a_1}(E_2), \ s\in  \bigcup_{t\in U_{r_1}(a_1)}U_{r_2}(a_2e^t)\Big\},
			\end{split}
		\end{equation*}
		where $\Psi_{r,a}$ is defined as in the statement.
		Hence, $(v)$ follows from the fact that
		\begin{equation*}
			\bigcup_{t\in U_{r_1}(a_1)}U_{r_2}(a_2e^t)=U_{r_1r_2}(a_1a_2).
		\end{equation*}
		Finally, by $(i)$,~\eqref{invariance}, and the change of variables $(n,a)=(n',1)\exp(tZ)$, one has
		\begin{equation*}
			\begin{split}
				\mu(P_{r,E}(a))&=\mu(P_{r,E}(1)\exp(\log a Z))=\mu(P_{r,E}(1))\\
				&=\int_{-\log r}^{\log r}\int_{E}\varphi((n,1)\exp{(tZ))}dn dt\\
				&=\int_{-\log r}^{\log r}dt\int_{E}\varphi(n,1)dn\\
				&=2\log r \mu_N (E),
			\end{split}    
		\end{equation*}
		which gives $(vi)$ and completes the proof.
	\end{proof}
	
	\subsection{Admissible cylinders}
	In order to introduce the family of admissible cylinders, we need to recall a celebrated result by Christ, which guarantees the existence of systems of dyadic cubes in any doubling metric space. We will apply the result to the doubling space $(N,d_N,\mu_N)$.

	\begin{thm}[\cite{christ1990}]\label{christ}
		Let $(N,d_N,\mu_N)$ be a doubling metric space. There exist a family $\mathcal Q:=\{Q\in\mathcal Q_k\colon k\in\mathbb Z\}$ of open sets of $N$ and constants $\delta\in (0,1)$, $C_1,c>0$, such that
		for each $k\in\zz$, $ \mathcal Q_k$ consists of countably many pairwise disjoint subsets of $N$ enjoying the following properties:
		\begin{enumerate}[label=(\roman*)]
			\item $\displaystyle \mu_N\Bigl(N\setminus \bigsqcup_{Q\in \mathcal Q_k} Q\Bigr)=0$, for every $k\in \mathbb{Z}$;
			\item for every $Q\in\mathcal{Q}_k$, there exists a unique set $ p_N(Q)\in\mathcal{Q}_{k-1}$ such that $Q\subset  p_N(Q)$, while $Q\cap Q'=\emptyset$ for any other $Q'\in \mathcal Q_{k-1}$;
			\item\label{eccentricityChrist} for each $Q\in\mathcal Q_k$ there exists a point $n_Q\in Q$ such that 
			\[B_N(n_Q,c\delta^k)\subset Q\subset B_N(n_Q,C_1\delta^k);\]
			\item\label{volumeChrist} for each $Q\in\mathcal Q$, $\mu_N( p_N(Q))\leq C_1\mu_N(Q)$ and if we put $s_N(Q):=\{Q'\in\mathcal{Q}\colon Q= p_N(Q')\}$, then $\# s_N(Q)\leq C_1$.
		\end{enumerate}
	\end{thm}
	
	We refer to the properties~$(iii)$ and~$(iv)$ as \textit{eccentricity condition} and \textit{volume control condition}, respectively. It is immediate that~$(iv)$ is a consequence of~$(iii)$, and $\# s_N(Q)\leq C_1$ follows by $\mu_N(p_N(Q))\leq C_1\mu(Q)$. 
	For simplicity and without loss of generality, we shall assume $C_1\geq 3$ and $\# s_N(Q)\geq 2$ for every $Q$. 
	Furthermore, 
	if $Q_1$, $Q_2$ are such that $p_N(Q_i)=Q$, then, by~$(vi)$ we have
	$\mu_N(Q_1)\leq \mu_N(Q)\leq C_1\mu_N(Q_2)$.
	Hence
	\[\frac{\mu_N(Q)}{\mu_N(Q_1)}=1+\sum_{\substack{Q'\in s_N(Q)\\Q'\neq Q_1}} \frac{\mu_N(Q')}{\mu_N(Q_1)}\geq1+\frac{\# s_N(Q) -1}{C_1}\geq 1+\frac 1{C_1},\]
	which gives us
	\begin{equation}\label{stima figli da sotto}
		\mu_N(Q)\geq\Bigl(1+\frac 1{C_1}\Bigr)\mu_N(Q'),\qquad Q'\in s_N(Q).
	\end{equation}
	
	In short, we say that $N$ admits a Christ decomposition, which from now on will be intended to be the family of sets $\mathcal{Q}$ prescribed by Theorem \ref{christ}. 
	We refer to these sets as \textit{Christ  cubes} and we say that $Q\in\mathcal Q_k$ has \textit{generation} $k\in\mathbb Z$. We will compare this result with the dyadic decomposition we provide in Theorem~\ref{thm dyadic partitions}.

	\begin{defn}\label{ammiss}
		Fix $\gamma\geq 5$, 
		and $\lambda>e^3/\delta$. We say that the cylinder $P=P_{r,Q}(a)$ is an admissible cylinder if $Q$ is a Christ cube of $N$ of generation $k$, for some $k\in\zz$, and one of the following holds:
		\begin{enumerate}
			\item $r> e$ and $ar^2\leq \delta^k\leq \lambda ar^\gamma$, and in this case we say that $P$ is a \textit{large} admissible cylinder;
			\item $1<r\leq e$ and $ae^2\log r\leq \delta^k\leq \lambda ae^2\log r$, and in this case we say that $P$ is a \textit{small} admissible cylinder.
		\end{enumerate}
	\end{defn}
	We now provide a canonical way to partition an admissible cylinder as the disjoint union of smaller admissible cylinders of comparable measure.
	
	\begin{defn}\label{def sons}
		Given an admissible cylinder $P=P_{r,Q}(a)$, we define the associated cylinders
		\begin{equation*}
			P^\vee =P_{\sqrt{r},Q}\Bigr(\frac a{\sqrt{r}}\Bigr), \ \ \ P^\wedge=P_{\sqrt{r},Q}(a\sqrt{r}).
		\end{equation*}
		Then, we define the set $s(P)$ of the \textit{sons} of $P$ as follows:
		\begin{enumerate}[label=$(\roman*)$]
			\item $s(P)=\{P^\vee,P^\wedge\}$, if $P^\vee$ and $P^\wedge$ are simultaneously admissible;
			\item $s(P)=\{P_{r,Q'}(a): \ Q'\in s_N(Q)\}$, otherwise.
		\end{enumerate}
	\end{defn}
	\begin{prop}\label{prop sons}
		Let $Z\in \mathcal{Z}$, $\mu\in \mathcal{F}_Z$ and $P$ be an admissible cylinder. Every $P'\in s(P)$ is admissible and 
		\begin{equation}\label{stimafigli}
			\Bigl(1+\frac 1{C_1}\Bigr)\mu(P')\leq \mu(P)\leq C_1 \mu(P').
		\end{equation}
	\end{prop}
	\begin{proof}
		Let $P=P_{r,Q}(a)$ be an admissible cylinder. We need to analyze separately different cases.
		
		\noindent\textsc{case 1}: $P$ is small admissible. In this case, $P^\vee$ and $P^\wedge$ cannot be large admissible. On the other hand, it is easy to see that they are simultaneously small admissible unless
		\begin{equation*}
			(i) \ \delta^k<\frac{\sqrt{r}}{2}ae^2\log r \quad \text{or} \quad (ii) \ \delta^k > \frac{1}{2\sqrt{r}}\lambda ae^2\log r.
		\end{equation*}
		But $(i)$ cannot hold, since it contradicts the small admissibility of $P$. It remains to check that if $(ii)$ holds, then $P_{r,Q'}(a)$ is small admissible for some (hence, for any) $Q'\in s_N(Q)$, i.e., that
		\begin{equation*}
			ae^2\log r\leq \delta^{k+1}\leq \lambda ae^2\log r.
		\end{equation*}
		The inequality on the right is implied by $P$ being small admissible (since $\delta^{k+1}<\delta^k$). For the other, one can use $(ii)$, $\delta\lambda>e^3$ and $r\leq e$ to see that
		\begin{equation*}
			\delta^{k+1} > \frac{\delta\lambda}{2\sqrt{r}}ae^2\log r > \frac{1}{2}ae^{9/2}\log r> ae^2\log r.
		\end{equation*}
		
		\noindent\textsc{case 2}: $P$ large admissible. In this case we need to distinguish two sub-cases.
		
		\textsc{case 2.1}: $r>e^2$. In this case, $P^\vee$ and $P^\wedge$ cannot be small admissible. On the other hand, it is easy to see that they are simultaneously large admissible unless
		\begin{equation*}
			(i)' \ \delta^k<ar^{3/2} \quad \text{or} \quad (ii)' \ \delta^k > \lambda a r^{(\gamma-1)/2}.
		\end{equation*}
		But $(i)'$ cannot hold, since it contradicts the large admissibility of $P$. It remains to check that if $(ii)'$ holds, then $P_{r,Q'}(a)$ is large admissible for some (hence, for any) $Q'\in s_N(Q)$, i.e., that
		\begin{equation}\label{eq: prop ammissibilità padri}
			ar^2\leq \delta^{k+1}\leq \lambda ar^\gamma.
		\end{equation}
		The inequality on the right is implied by $P$ being large admissible (since $\delta^{k+1}<\delta^k$). For the other, one can use $(ii)'$, $\delta\lambda>e^3$ and $\gamma>5$ to see that
		\begin{equation*}
			\delta^{k+1} > \delta\lambda a r^{(\gamma-1)/2}>e^3ar^2>ar^2.
		\end{equation*}
		
		\textsc{case 2.2}: $e<r\leq e^2$. In this case, $P^\vee$ and $P^\wedge$ cannot be large admissible, and we know from \textsc{case 1} that they are simultaneously small admissible unless either $(i)$ or $(ii)$ hold. But once again, $(i)$ cannot hold. Indeed, since $\sup_{(e,e^2]}r^{-3/2}\log r=e^{-3/2}$, we have

		\begin{equation*}
			\frac{\sqrt{r}}{2}ae^2\log r= \frac{e^2}{2}\bigl(r^{-3/2}\log r\bigr)ar^2\leq \frac{\sqrt{e}}{2}ar^2\leq ar^2,
		\end{equation*}
		which makes it clear that $(i)$ contradicts the large admissibility of $P$. It remains to check that if $(ii)$ holds, then $P_{r,Q'}(a)$ is large admissible for some (hence, for any) $Q'\in s_N(Q)$, i.e., that \eqref{eq: prop ammissibilità padri} holds true.
		
		The inequality on the right of \eqref{eq: prop ammissibilità padri} is implied by $P$ being large admissible (since $\delta^{k+1}<\delta^k$). For the left inequality, one can use $(ii)$, $\delta\lambda>e^3$ and the fact that $\inf_{(e,e^2]}r^{-5/2}\log r=2e^{-5}$ to get
		
		\begin{equation*}
			\delta^{k+1} > \frac{\delta\lambda}{2\sqrt{r}}ae^2\log r > \frac{1}{2\sqrt{r}}ae^5\log r= \frac{e^5}{2}\bigl(r^{-5/2}\log r\bigr) ar^2\geq ar^2.
		\end{equation*}
		Finally,~\eqref{stimafigli} follows by $(vi)$ in Proposition~\ref{prop-cyl},~\eqref{stima figli da sotto}, Theorem~\ref{christ}~$(iv)$, and the fact that $C_1\geq 3$.
	\end{proof}

	\begin{defn}\label{def parent}  
		Given an admissible cylinder $P=P_{r,Q}(a)$, we define the associated cylinders
		\begin{equation*}
			p^\downarrow(P)=P_{r^3,Q}\left( {\frac a {r^2}}\right), \quad p^\uparrow(P)=P_{r^2,Q}(ar), \quad p^\leftrightarrow(P)=P_{r,p_N(Q)}(a).
		\end{equation*}
	\end{defn}

	\begin{prop}\label{ammissibilita padri}
		Let $P=P_{r,Q}(a)$ be an admissible large cylinder, $r>e$, $Q\in\mathcal Q_k$, $a\in A$. Then,
		\begin{enumerate}
			\item if $ar^2\leq \delta^k\leq \delta\lambda ar^\gamma$, then $p^\leftrightarrow(P)$ is a large admissible cylinder;
			\item  if $ \delta\lambda ar^\gamma< \delta^k\leq \lambda ar^\gamma$, then $p^\downarrow(P)$, $p^\downarrow(P)\setminus P$, $p^\uparrow(P)$, $p^\uparrow(P)\setminus P$ are large admissible cylinders.
		\end{enumerate}
		Furthermore, whenever $\tilde P\in\{p^\leftrightarrow(P),p^\downarrow(P), p^\uparrow(P)\}$,
		\[\Bigl(1+\frac 1{C_1}\Bigr)\mu(P)\leq\mu(\tilde P)\leq C_1\mu(P).\]
	\end{prop}
	
	\begin{proof}
		Let $P=P_{r,Q}(a)$ be an admissible large cylinder, $Q\in\mathcal Q_k$.
		If $ar^2\leq \delta^k\leq \delta\lambda ar^\gamma$, then $p^\leftrightarrow(P)=P_{r,p_N(Q)}(a)$, with $p_N(Q)\in\mathcal Q_{k-1}$, is large admissible since
		\[ar^2\leq \delta^k<\delta^{k-1}\leq \lambda ar^\gamma.\]
		This proves~(1). If $ \delta\lambda ar^\gamma< \delta^k\leq \lambda ar^\gamma$, then the fact that $p^\downarrow(P)$ and $p^\uparrow(P)$ are large admissible follows from the fact that $\delta\lambda>e^3$ and $\gamma\geq 5$. Indeed, we have
		\[ar^4<\delta\lambda ar^\gamma<\delta^k\leq\lambda ar^\gamma<\lambda a r^{3\gamma-2},\quad ar^5<\delta\lambda ar^\gamma<\delta^k\leq\lambda ar^\gamma<\lambda a r^{2\gamma+1}.\]
		Note that 
		\[p^\downarrow(P)\setminus P=P_{r^2,Q}\Bigl(\frac a{r^3}\Bigr),\qquad     p^\uparrow(P)\setminus P=P_{r,Q}(ar^2), \]
		and, as above, they are large admissible because, by $\delta\lambda>e^3$ and $\gamma\geq 5$
		\[ar<\delta\lambda ar^\gamma<\delta^k\leq\lambda ar^\gamma<\lambda a r^{2\gamma-3},\quad ar^5<\delta\lambda ar^\gamma<\delta^k\leq\lambda ar^\gamma<\lambda a r^{\gamma+2}.\]
		Hence, we proved~(2). The statement on the measures follows by Proposition~\ref{prop-cyl}~$(vi)$, \eqref{stima figli da sotto}, and Theorem~\ref{christ}~$(iv)$. Indeed,
		\[\Bigl(1+\frac 1{C_1}\Bigr)\mu(P)\leq\mu(p^\leftrightarrow(P))=2\log r\mu_N(p_N(Q))\leq C_1 \mu(P),\] \[\mu(p^\downarrow(P))=3\mu(P),\quad \mu(p^\uparrow(P))=2\mu(P).\]
		
	\end{proof}

	\subsection{Dyadic and \CZ decompositions}
	We start by constructing a dyadic decomposition of the measure space $(G,\mu)$ made of admissible sets.

	\begin{thm}\label{thm dyadic partitions}
		For any $Z\in \mathcal{Z}$ and any $\mu\in\mathcal{F}_Z$, there exists a family $\mathcal{D}^Z=\{P\in\mathcal{D}^Z_k\colon k\in\mathbb Z\}$ such that
		for each $k\in\zz$, $\mathcal{D}^Z_k$ consists of pairwise disjoint admissible cylinders enjoying the following properties:
		\begin{enumerate}
			\item[$(i)$] $\displaystyle \mu\Bigl(G\setminus\bigsqcup_{P\in\mathcal{D}^Z_k} P\Bigr)=0$, for every $k\in \mathbb{Z}$;
			\item[$(ii)$] for every $P\in\mathcal{D}^Z_k$, there exists a unique cylinder $p(P)\in\mathcal{D}^Z_{k-1}$ such that $P\subset p(P)$, while $P\cap P'=\emptyset$ for any other $P'\in \mathcal{D}^Z_{k-1}$
			;
			\item[$(iii)$] for almost every $x\in G$ there is a, necessarily unique, cylinder  $P_k^x \in\mathcal{D}^Z_k$ which contains $x$  for any $k\in\zz$ and for any such $x$ 
			\[\lim_{k\rightarrow -\infty}\mu(P_k^x)=+\infty,\qquad\lim_{k\rightarrow +\infty}\mu(P_k^x)=0 ;\]
			\item[$(iv)$] 	for every $P\in\mathcal{D}^Z$, $\mu(p(P))\leq C_1\mu(P)$ and $\#\{P'\in\mathcal{D}^Z\colon P=p(P')\} \leq C_1$.
		\end{enumerate}
	\end{thm}
	
	Observe that~$(iv)$ above essentially extends the volume control condition ($(iv)$ of Theorem~\ref{christ}) from the group $N$ to the group $G$. However, while in doubling metric spaces the volume control condition is a trivial consequence of the eccentricity condition ($(iii)$ of Theorem~\ref{christ}),
	this is not the case, in general, in nondoubling metric spaces and, a fortiori, in our case where the eccentricity condition, and more in general any metric condition, is not even available at all. Hence, having the volume control condition $(iv)$ in Theorem \ref{thm dyadic partitions} is noteworthy and it has to be proved from scratch. Our construction is inspired by the one developed for $ax+b$ groups equipped with a right Haar measure in \cite{liuvallarinoyang}.

	\begin{proof}
		Fix a large admissible cylinder $P_0=P_{r_0,Q_0}(1)$ and let $P_{k+1}:=p(P_{k})$. Here we are defining $p(P_k)$ as follows: $p(P_k)=p^\leftrightarrow(P_k)$ whenever admissible, then for the smallest $k$ such that the latter is not admissible we set $p(P_k)=p^\uparrow(P_k)$, for the next value of $k$ such that $p^\leftrightarrow(P_k)$ is not admissible we set $p(P_k)=p^\downarrow(P_k)$, and then we keep on alternating $p^\uparrow$ and $p^\downarrow$ in this way.
		
		We set $P_k=P_{r_k,Q_k}(a_k)$ for some $r_k>e$, $j(k)\in\mathbb Z$, $Q_k\in\mathcal{Q}_{j(k)}$, $a_k\in A$. 
		By Proposition~\ref{ammissibilita padri}, it can neither happen that $p^n(P_k)=p^\leftrightarrow(p^{n-1}(P_{k}))$ for all $n\in\mathbb N$ nor that $p^n(P_k)\neq p^\leftrightarrow(p^{n-1}(P_{k}))$ for all $n\in\mathbb N$. Hence there is an alternation of all the three choices that makes $r_k\to +\infty$ and $j(k)\to -\infty$ as $k\to +\infty$, so that 
		\[G=\bigcup_{k\in\mathbb N}P_k.\]
		
		Now we introduce the following notation: if $P=P_{r,Q}(a)$, where $Q\in\mathcal Q_j$, then the family of its \textit{siblings} is
		\[\mathcal{S}(P)=\{P_{r,Q'}(a): Q' \in\mathcal Q_j\}.\]
		For every $k\in \mathbb N$, we define the families
		$\mathcal S_k:=\mathcal S(P_k)$ and
		\begin{equation*}
			\widetilde{\mathcal S}_k:=\begin{dcases}
				\mathcal{S}(P_{k+1}\setminus P_k),&\text{if }P_{k+1}\in\{p^\uparrow(P_k),p^\downarrow(P_k)\},\\
				\emptyset,&\text{if }P_{k+1}=p^\leftrightarrow(P_k).
			\end{dcases}
		\end{equation*}
		By Proposition~\ref{ammissibilita padri}, all the cylinders in $\mathcal S_k$ and $\widetilde{\mathcal S}_k$ are large admissible. Furthermore, the sets in $\mathcal S_k\cup (\bigcup_{\ell\geq k}\widetilde{\mathcal S}_\ell)$
		are disjoint and their union has full measure in $G$ with respect to $\mu$. 
		
		Fix $k\in\mathbb{N}$ and take $\ell>k$ such that $	\widetilde{\mathcal S}_\ell\neq\emptyset$. For every $P\in	\widetilde{\mathcal S}_\ell$, we iterate Definition~\ref{def sons}, by putting $s^1(P)=s(P)$ and $s^{m+1}(P):=s(s^{m}(P))$, for every $m\in\mathbb{N}$, $m\geq 1$, and we put
		\[\widetilde{\mathcal S}_\ell^m:=\{P'\in s^m(P)\colon P\in \widetilde{\mathcal S}_\ell\}.\]
		Clearly, 
		the disjoint union of the sets in $\widetilde{\mathcal{S}}_\ell^m$
		has full measure in the union of the sets in $\widetilde{\mathcal S}_\ell$ with respect to $\mu$.
		We define
		\[\mathcal D^Z_{-k}:=\mathcal S_k\cup \widetilde{\mathcal S}_k\cup\Bigl(\bigcup_{\ell>k}\widetilde{\mathcal S}_\ell^{\ell-k}\Bigr).\]
		By construction, the sets in $\mathcal D^Z_{-k}$ are disjoint and their union has full measure in $G$ with respect to $\mu$. 
		For $P\in \mathcal S_k\subset \mathcal D^Z_{-k}$, we define $p(P)$ to be $p^\leftrightarrow(P), p^\uparrow(P)$ or $p^\downarrow(P)$, when $p(P_k)$ is given by $p^\leftrightarrow(P_k), p^\uparrow(P_k)$ or  $p^\downarrow(P_k)$, respectively. Observe that $p(P)\in \mathcal{S}_{k+1}\subset\mathcal D^Z_{-k-1}$.
		When $P\in \widetilde{\mathcal S}_k$, there is $P'\in \mathcal S_k$ such that $P=p(P')\setminus P'$ and then we put $p(P):=p(P')\in \mathcal S_{k+1}\subset\mathcal D^Z_{-k-1}$. In both cases $\mu(p(P))\leq C_1\mu(P)$ by Proposition~\ref{ammissibilita padri}.
		Finally, if $P\in\widetilde{\mathcal{S}}_\ell^{\ell-k}$ for some $\ell>k$, then we write $p(P)$ for the unique set in $\widetilde{\mathcal{S}}_\ell^{\ell-k-1}\subset\mathcal D^Z_{-k-1}$ such that $P\in s(p(P))$. The control of the measures follows by Proposition~\ref{prop sons}. The properties~$(i)$, $(ii)$ and $(iv)$ have been proved to be satisfied by sets in $\{\mathcal D^Z_{-k}\}_{k\in\mathbb N}$, since the last inequality in~$(iv)$ follows by the control of the measure, as already observed. Coming to the ``positive generations'', we define inductively
		\[\mathcal D^Z_k:=\{P'\in s(P)\colon P\in\mathcal D^Z_{k-1}\},\qquad k\in\mathbb{N}.\]
		For every $P'\in\mathcal D^Z_{k}$, we put $p(P')=P$ where $P\in\mathcal D^Z_{k-1}$ is such that $P'\in s(P)$.
		Note that $(i)$, $(ii)$ and $(iv)$ follow by the construction above and by Proposition~\ref{prop sons}. 
		
		We showed that 
		\[\mu(V_k)=0,\qquad V_k:=G\setminus\bigcup_{P\in\mathcal D^Z_k}P,\qquad k\in\mathbb Z. \]
		Furthermore, $V_k\subset V_{k+1}$ and then 
		\[\mu(V)=\lim_{k\to+\infty}\mu(V_k)=0,\qquad V:=\bigcup_{k\in\mathbb Z}V_k.\]
		It remains to prove $(iii)$. Let $x\in G\setminus V$. For every $k\in\mathbb Z$, there exists $P_k^x\in \mathcal D^Z_k$ such that $x\in P_k^x$.
		Clearly, $\{P_k^x\}_{k\in\mathbb N}$ is a sequence of elements of $\mathcal D^Z$ such that $p(P_k^x)=P_{k-1}^x$, then by Propositions~\ref{prop sons} and~\ref{ammissibilita padri}, we have that
		\[\mu(P_{k-1}^x)\geq\Bigl(1+\frac 1{C_1}\Bigr)\mu(P_k^x).\]
		Hence, since $1+\frac 1{C_1}>1$, by iterating the inequality we prove $(iii)$.
	\end{proof}

	A natural consequence of Theorem \ref{thm dyadic partitions} is that $L^1(\mu)$ admits a CZ decomposition with respect to the family $\mathcal{D}^Z$. The proof follows classical lines, but we include it here for the convenience of the reader.
	
	\begin{thm}\label{CZdec}
		For any $Z\in\mathcal{Z}$ and any $\mu\in\mathcal{F}_Z$, the space $L^1(\mu)$ admits a CZ decomposition with respect to the family $\mathcal{D}^Z$. Namely, for every $f\in L^1(\mu)$ and $\alpha>0$ there exists a family of disjoint sets $\mathcal{E}(f,\alpha)=\{P_j\}$, with $P_j\in \mathcal D^Z$ and functions $g,b_j$ such that  $f=g+ {\sum_j} b_j$ and
		\begin{enumerate}[label={\rm (\alph*)}]
			\item $|g| \le C_1 \alpha$ $\mu$-almost everywhere;
			\item  $b_j=0$ on $G \setminus P_j$ and $\displaystyle\int_{P_j} \ b_j \ d\mu = 0$;
			\item $\displaystyle\sum_j \|b_j\|_1 \le 2C_1\|f\|_1$.
		\end{enumerate}
		Conversely, any set $P\in\mathcal{D}^Z$ belongs to $\mathcal{E}(f,\alpha)$ for some $f\in L^1(\mu)$ and some $\alpha>0$.
	\end{thm}
	
	\begin{proof}
		Let $f\in L^1(\mu)$, and consider the dyadic Hardy--Littlewood maximal operator associated with $\mathcal D^Z$ applied to $f$, i.e.,
		$$\mathcal M^{\mathcal D}_Z f(x)=\sup_{P\in\mathcal D^Z, P\ni x}\frac{1}{\mu(P)}\int_P|f|d\mu,\qquad x\in G.
		$$
		Let $\alpha>0$ and $E_{\alpha}=\{x\in G: \mathcal M^{\mathcal D}_Z f(x)>\alpha \}$. Let $V\subset G$ be the null measure set as in the proof of Theorem~\ref{thm dyadic partitions}. For every $x\in E_{\alpha}\cap V$, let $P^x$ be the set of maximal measure in $\mathcal D^Z$ among those containing $x$ such that 
		\[\frac{1}{\mu(P^x)}\int_{P^x}|f|d\mu>\alpha.\] 
		Such set exists because, by Theorem~\ref{thm dyadic partitions}~$(iii)$, if $x\in P^x_k\in \mathcal{D}^Z_k$ then $\mu(P_k^x)\to+\infty$ as $k\to -\infty$. The family $\{P^x\}_{x\in E_{\alpha}}$ is at most countable and we denote it by $\{P_j\}$. We then have that ${\bigsqcup\limits_j}\,P_j$ has full measure in $E_\alpha$ and
		\begin{equation}\label{proof c}
			\mu(E_\alpha)=\sum_j \mu(P_j)\leq \frac{1}{\alpha}\sum_j\int_{P_j}|f|d\mu \leq \frac{1}{\alpha}\|f\|_1.
		\end{equation}
		Moreover, for every $j$,
		\begin{equation}\label{maximal dyadic auxiliary equation}
			\frac{1}{\mu(P_j)}\int_{P_j}|f|d\mu>\alpha,\quad\text{and}\quad \frac{1}{ {\mu(p(P_j))}}\int_{p(P_j)}|f|d\mu\leq\alpha.
		\end{equation}
		We define now 
		\begin{align*}
			&g(x) = \begin{cases}  \displaystyle\frac{1}{\mu(P_j)} \int_{P_j} f \ d\mu,  &\text{if $x \in P_j$,} \\ 
				f(x), &\text{else,}\end{cases} \\ 
			&b_j(x) = \bigg(f(x)- \frac{1}{\mu(P_j)}\int_{ P_j} f \ d\mu \bigg) \chi_{P_j}(x).  
		\end{align*} 
		Clearly (b) holds. Moreover by \eqref{maximal dyadic auxiliary equation} and Theorem \ref{thm dyadic partitions}~$(iv)$,
		\begin{align}\label{proof d}
			\|b_j\|_{1} \le 2 \int_{P_j} |f|\ d\mu \le 2\int_{{p(P_j)}}|f|\ d\mu\le 2\alpha \mu( {p(P_j)})\le 2 {C_1}\alpha\mu(P_j).
		\end{align}
		Item (c) follows from~\eqref{proof c} and~\eqref{proof d}. To prove (a), assume first that $x\in E_\alpha$. Then, by \eqref{maximal dyadic auxiliary equation} and Theorem \ref{thm dyadic partitions} we have
		\begin{equation*}
			|g(x)|\leq\frac{1}{\mu(P_j)} \int_{P_j} |f| \ d\mu\leq \frac{C_1}{ {\mu(p(P_j))}}\int_{p(P_j)}|f|d\mu\leq C_1\alpha.
		\end{equation*}
		On the other hand, since for almost every $x$ we have the pointwise bound $|f|\leq \mathcal M^{\mathcal D}_Z f$, it follows that for almost every $x\notin E_\alpha$,
		\begin{equation*}
			|g(x)|=|f(x)|\leq \mathcal M^{\mathcal D}_Z f(x)\leq \alpha. 
		\end{equation*}
		This proves that for any $f\in L^1(\mu)$ and any $\alpha>0$ we can choose $\mathcal{E}(f,\alpha)\subset \mathcal{D}^Z$ with the desired properties, and therefore $\mathcal{E}=\{\mathcal{E}(f,\alpha): f\in L^1(\mu),\alpha>0\}\subset \mathcal{D}^Z$.
		
		Now we show that, conversely, for every $P\in \mathcal{D}^Z$ there exists a function $f\in L^1(\mu)$ and a number $\alpha>0$ such that $P\in \mathcal{E}(f,\alpha)$, which proves that $\mathcal{E}=\mathcal{D}^Z$.
		To see this, fix an arbitrary set $P_0\in \mathcal{D}^Z$, let $S\in s(P_0)$ and let $R$ be the set of minimal measure in $\mathcal{D}^Z$ among those properly containing $P_0$. Set $f=\chi_S$ and $\alpha=\mu(S)/\mu(R)$. Then,
		\begin{equation*}
			M^{\mathcal D}_Z f(x)=\sup_{P\in\mathcal D^Z, P\ni x}\frac{\mu(P\cap S)}{\mu(P)}.
		\end{equation*}
		Hence, if $x\in S$ we have $M^{\mathcal D}_Z f(x)=1>\alpha$, while for $x\not\in S$, denoting by $P'$ the smallest element of $\mathcal{D}^Z$ containing both $x$ and $S$, we have $ M^{\mathcal D}_Z f(x)=\mu(S)/\mu(P')$. But according to Theorem \ref{thm dyadic partitions}, the smallest set in $\mathcal{D}^Z$ containing $S$ is $P_0$. It follows that $M^{\mathcal D}_Z f(x)=\mu(S)/\mu(P_0)>\alpha$ if $x\in P_0\setminus S$, while $M^{\mathcal D}_Z f(x)\leq\mu(S)/\mu(R)=\alpha$ if $x\not\in P_0$. It follows that $E_{\alpha}=\{x\in G: \mathcal M^{\mathcal D}_Z f(x)>\alpha \}=P_0$. Following verbatim the above construction, one obtains a CZ decomposition for $f$ where the family $\{P_j\}$ consists of a single element, which is $P_0$.
	\end{proof}

	\section{The maximal Hardy--Littlewood operator}\label{sec: maximal}

	Along the lines of the proof of Theorem \ref{CZdec} in the previous section it is proved that the dyadic Hardy--Littlewood maximal function is of weak type $(1,1)$. In this section we provide a covering lemma for admissible cylinders which allows us to deduce that also the Hardy--Littlewood maximal function associated to the family of all admissible cylinders (not necessarily dyadic) is of weak-type $(1,1)$.
	
	This result fits into a quite active line of research, since in the last years many authors investigated the $L^p$ and weak type $(1,1)$ boundedness of Hardy--Littlewood maximal operators in nondoubling metric measure spaces, such as Lie groups and manifolds of exponential growth \cite{aldaz,  GGHM, GGM, GaudrySjogren, Ionescu, stromberg, Val}, nondoubling infinite graphs \cite{CMS10, HS, LMSV, LS, ORIS, SoriaTradacete}, and, more generally, nondoubling measure metric spaces \cite{Dariusz, Dariusz2, NaorTao, stempak}.

	\begin{defn}\label{def env}
		Given a constant $C>1$ and a cylinder $P=P_{r,Q}(a)$ its $C-$\textit{envelope} is defined as 
		\begin{equation*}
			\dil{P}{C}:=P_{r^C,Q}(a).
		\end{equation*}
	\end{defn}
	Fix $Z\in\mathcal Z$ and $\mu\in\mathcal F_Z$. Observe that by Proposition~\ref{prop-cyl}~$(vi)$ 
	\begin{equation}\label{envelope measure}
		\mu(\dil{P}{C})=C\mu(P).
	\end{equation}
	It is well known that in a metric measure space of homogeneous type the Vitali covering Lemma implies the weak type $(1,1)$ boundedness of the Hardy--Littlewood maximal function, see for example \cite[Theorem 2.1]{coifman}. Hence, it is natural in our context to check whether one can obtain appropriate covering lemmas for admissible cylinders which, together with \eqref{envelope measure},   imply the weak type $(1,1)$ boundedness for the Hardy--Littlewood maximal function
	\begin{equation*}   \mathcal{M}_Zf(x)=\sup_{\substack{P\in\mathcal P^Z\\ P\ni x}}\frac{1}{\mu(P)}\int_P|f|d\mu, \qquad f\in L^1_{\text{loc}}(\mu), \ x\in G,
	\end{equation*}
	where $\mathcal P^Z$ denotes the family of all admissible cylinders.
	
	In this section we carry out this program and we are able to prove the following maximal theorem.
	
	\begin{thm}\label{maximal theorem}
		For any $Z\in \mathcal{Z}$, $\mu\in\mathcal F_Z$, $\alpha>0$ and $f\in L^1(\mu)$,
		\begin{equation*}
			\mu(\{x\in G: \ \mathcal{M}_Zf(x)>\alpha\})\leq \frac{C_2}{\alpha}\Vert f\Vert_1,
		\end{equation*}
		with $C_2=3\max\{\gamma+1+\log\lambda, \lambda e^3\}$.
	\end{thm}
	The next result is a fundamental step in obtaining a covering lemma for admissible cylinders, from which Theorem \ref{maximal theorem} will follow as a rather direct consequence.

	\begin{lem}\label{intersecting cylinders}
		Let $P_j=P_{r_j,Q_j}(a_j)$, $j=1,2$, be two admissible cylinders such that $P_1\cap P_2\neq\emptyset$. Denote by $k_j$ the generation of $Q_j$. If $k_1\leq k_2$ then $P_2\subset \dil{P_1}{C_2}.$
	\end{lem}
	
	\begin{proof}
		Let $I_j:=U_{r_j}(a_j)$, defined as in~\eqref{Ur}. By Proposition \ref{prop-cyl}~$(iv)$, $Q_1$ and $Q_2$ must intersect, hence due to the properties of dyadic sets and $k_2\geq k_1$ it must be $Q_2\subset Q_1$, and $I_1\cap I_2\neq \emptyset$, hence
		\begin{subnumcases}{\left(\frac{a_1}{r_1},a_1r_1\right)\cap \left(\frac{a_2}{r_2},a_2r_2\right)\neq \emptyset\quad\text{i.e.}\quad }
			a_1< a_2r_1r_2,\label{equazione}\\
			a_2< a_1r_1r_2.\label{equazione2}
		\end{subnumcases}
		For every interval $I=(t-R,t+R)\subset\R$ and $C>0$ we put $I^C:=(t-CR,t+CR)$. Hence it is enough to prove that $I_2\subset I_1^{C_2}=U_{r_1^{C_2}}(a_1)$, or equivalently,
		that
		\begin{equation*}
			\log (ar_2)\leq {C_2} \log r_1, \quad a:=\max\Big\{\frac{a_1}{a_2},\frac{a_2}{a_1}\Big\}.
		\end{equation*}
		If $r_1\geq r_2$ the result is easily proved: by~\eqref{equazione} and~\eqref{equazione2}, we have
		\[\log(ar_2)< \log(r_1r_2^2)\leq 3\log(r_1).\]
		Hereinafter, let $r_1<r_2$. By the same argument we have $\log(ar_2)\leq 3\log(r_2)$.
		To prove the result it is therefore sufficient to show that $3\log r_2\leq {C_2}\log r_1$. In order to do that, we distinguish three cases. If $P_2$ is small then also $P_1$ is, and 
		\begin{equation*}
			a_2e^2\log r_2\leq\delta^{k_2}\leq \delta^{k_1}\leq\lambda a_1 e^2\log r_1,
		\end{equation*}
		which together with \eqref{equazione} gives
		\begin{equation*}
			\log r_2<\lambda r_1r_2 \log r_1\leq \lambda e^2\log r_1.
		\end{equation*}
		If $P_1$ and $P_2$ are both large, then
		\begin{equation*}
			a_2r_2^2\leq\delta^{k_2}\leq \delta^{k_1}\leq \lambda a_1r_1^\gamma,
		\end{equation*}
		which together with \eqref{equazione}, and the fact that, since $r_1>e$, $\lambda<r_1^{\log \lambda}$, gives
		\begin{equation*}
			r_2< \lambda r_1^{\gamma+1}<r_1^{\gamma+1+\log \lambda},
		\end{equation*}
		which in turn implies
		\begin{equation*}
			\log r_2\leq (\gamma+1+\log \lambda)\log r_1.
		\end{equation*}
		Finally, if $P_2$ is large and $P_1$ is small, then
		\begin{equation*}
			a_2r_2^2\leq\delta^{k_2}\leq \delta^{k_1}\leq\lambda a_1e^2\log r_1,
		\end{equation*}
		which together with \eqref{equazione} gives
		\begin{equation*}
			\log r_2\leq r_2< \lambda r_1 e^2 \log r_1\leq \lambda e^3 \log r_1.
		\end{equation*}
		Summing up, in any case, if $P_1$ and $P_2$ are admissible then 
		\begin{equation*}
			3\log r_2\leq {C_2}\log r_1,
		\end{equation*}
		where ${C_2}=3\max\{\gamma+1+\log\lambda, \lambda e^3\}$, since $\lambda e^3>1$, and this concludes the proof.
	\end{proof}
	
	From this technical geometric result, a Vitali-like covering lemma for admissible cylinders follows.

	\begin{lem}[Covering lemma for admissible cylinders]\label{vitali}
		Let $\mathcal{C}$ be a family of admissible cylinders such that 
		\[k_0:=\min\{k\in\mathbb Z\colon P_{r,Q}(a)\in\mathcal C,\,Q\in\mathcal{Q}_k\}>-\infty.\]
		Then, there exists a countable subfamily $\mathcal{G}\subset\mathcal{C}$ such that the cylinders of $\mathcal{G}$ are pairwise disjoint and for each $P\in \mathcal{C}$ there exists $R\in \mathcal{G}$ with $P\cap R\neq \emptyset$ and $P\subset \dil{R}{{C_2}}$. In particular,
		\begin{equation*}
			\bigcup_{P\in \mathcal{C}}P\subset\bigcup_{R\in \mathcal{G}}\dil{R}{C_2}.
		\end{equation*}
	\end{lem}
	
	\begin{proof}
		Choose a cylinder $P_0\in\mathcal C$ with base set of generation $k_0$. For every $i\geq 1$, among all the sets of $\mathcal{C}$ not intersecting 
		$P_j$ for any $j<i$ (if any exists), let $P_i$ be one with base set of minimal generation, say $k_i$. The elements of the resulting family $\mathcal{G}=\{P_i\}_{i=0}^\infty$ are pairwise disjoint. Moreover, for any $P\in\mathcal{C}$ there exists an index $i$ such that $P\cap P_i\neq\emptyset$. In particular, if $P_i$ is a cylinder with base set of minimal generation among those of $\mathcal{G}$ intersecting $P$, then the generation of the base set of $P$ must be $\geq k_i$. In fact, either $P\in \mathcal{G}$, in which case it must be $P=P_i$, or $P\not\in \mathcal{G}$, and since $P$ does not intersect elements of $\mathcal{G}$ having base sets of generation $<k_i$, the generation of its base set cannot be $<k_i$, since otherwise $P$ would belong to $\mathcal{G}$. Hence, by Proposition \ref{intersecting cylinders}, $P\subset \dil{P_i}{{C_2}}$.
	\end{proof}
	We are now in a position to prove Theorem~\ref{maximal theorem}.
	
	\begin{proof}[Proof of Theorem \ref{maximal theorem}]
		
		Let $x\in E_\alpha=\{x\in G: \ \mathcal{M}_Zf(x)>\alpha\}$, and choose an admissible cylinder $P_x$ containing $x$ such that 
		\[\int_{P_x}|f|\de\mu>\alpha\mu(P_x).\]
		Let $K$ be a compact subset of $E_\alpha$. Then $\{P_x\colon x\in K\}$ is a covering of $K$. Since $K$ is compact, there exists a subcovering $\mathcal{C}$ of $K$ made by a finite family of admissible cylinders. Since the family $\mathcal{C}$ is finite, we can extract from it a subfamily $\mathcal{G}$ with the properties prescribed by Lemma \ref{vitali}. Then,
		\begin{equation*}
			\Vert f\Vert_1\geq \sum_{P\in \mathcal{G}}\int_P |f| d\mu\geq \alpha \sum_{P\in \mathcal{G}}\mu(P)=\frac{\alpha}{{C_2}}\sum_{P\in \mathcal{G}}\mu(\dil{P}{C_2})\geq \frac{\alpha}{{C_2}}\mu(K).
		\end{equation*}
		By the inner regularity of the measure, passing to the supremum over all compact subsets $K$ of $E_\alpha$ we obtain the desired result.
	\end{proof}

	\section{The CZP for the \texorpdfstring{distance $d_Z$}{the flow distance}}\label{sec: CZP vertical flow}
	
	In this section we prove our main result, Theorem \ref{main}, by showing that for any $Z\in\mathcal{Z}$ and any $\mu\in\mathcal{F}_Z$ the family $\mathcal D^Z$ of dyadic admissible cylinders constructed in Theorem \ref{thm dyadic partitions} is a CZ family for $(G,d_Z,\mu)$. Here $d_Z$ denotes a metric on $G$ which takes into account the action of the vector field $Z$, defined by
	\begin{equation}\label{distZ}
		d_Z((n,a),(n',a')):=d_G\left(\bigl(nn\left(\log a\right)^{-1},a\bigr), \bigl(n'n\left(\log a'\right)^{-1},a'\bigr)\right),  \end{equation}
	for every $(n,a),(n',a')\in G$. We call this metric a $Z-$flow metric.
	
	Observe that the symmetry of $d_Z$ follows from the symmetry of $d_G$, as well as the triangular inequality. If $d_Z((n,a),(n',a'))=0$, then $a=a'$ and
	\[nn\left(\log a\right)^{-1}=n'n\left(\log a'\right)^{-1}=n'n\left(\log a\right)^{-1}, \]
	which implies $n=n'$.
	
	We shall denote by $B_G^Z(x,R)$ the ball centered at $x\in G$ with radius $R>0$ with respect to $d_Z$.
	
	\begin{lem}\label{lem 1}
		There exists a constant $C_3>0$ such that for every admissible cylinder $P=P_{r,Q}(a)$,
		\begin{equation*}
			P\subset B_G^Z\bigl((n_Q,1)\exp((\log a)Z),C_3\log r\bigr), 
		\end{equation*}
		where $n_Q$ is defined in Theorem~\ref{christ}~$(iii)$.
	\end{lem}
	\begin{proof}
		Let $x=(n,1)\exp(tZ)$ be a point in $P$. By~\eqref{CC distance} and the fact that $e^t\in (a/r,ar)$, we get
		\begin{align*}    
			\cosh&\Bigl(d_Z\bigl(x,(n_Q,1)\exp((\log a)Z)\bigr)\Bigr)=\cosh\Bigl(d_G\bigl((n,e^t),(n_Q,a)\bigr)\Bigr)\\
			&=\cosh\left(\log\frac{e^t}{a}\right)+\frac{1}{2ae^t}d_N\left(n,n_Q\right)^2<\cosh(\log r)+\frac{d_N\left(n,n_Q\right)^2}{2a^2}  r.
		\end{align*}
		If $r>e$, since $P$ is admissible
		$d_N\left(n,n_Q\right)\leq C_1\delta^k\leq C_1 \lambda ar^\gamma$, by Theorem~\ref{christ}~$(iii)$. It follows that there exists $C>0$ such that for every $r>e$
		\begin{align*}    
			\cosh\Bigl(d_Z\bigl(x,(n_Q,1)\exp((\log a)Z)\bigr)\Bigr) {<\cosh(\log r)+\frac{C_1^2\lambda^2}{2}r^{2\gamma+1}\leq C} r^{2\gamma {+1}},
		\end{align*}
		{and then 
			\begin{align*}    
				d_Z\bigl(x,(n_Q,1)\exp((\log a)Z)\bigr)< {\rm arcosh}(C r^{2\gamma}\bigr)\leq \log(2C)+(2\gamma+1)\log r\lesssim\log r.
			\end{align*}
			If $1<r\leq e$, since $P$ is admissible $d_N\left(n,n_Q\right)\leq C_1\delta^k\leq C_1 \lambda ae^2\log r$, by Theorem~\ref{christ}~$(iii)$, 
			there exists $C>0$ such that 
			\begin{align*}
				\cosh\Bigl(d_Z\bigl(x,(n_Q,1)\exp((\log a)Z)\bigr)\Bigr)&< \cosh(\log r)+\frac{C_1^2e^4\lambda^2}{2}(\log r)^2r\\&\leq \cosh(C\log r).
			\end{align*} 
			So we proved}
		\begin{align*}    
			d_Z\bigl(x,(n_Q,1)\exp((\log a)Z)\bigr)\lesssim \log r,
		\end{align*}
		as desired.
	\end{proof}

	\begin{lem}\label{lem 2}
		Let $P=P_{r,Q}(a)$ be an admissible cylinder. We have that
		\begin{equation}\label{lem 7.1}
			P^*:=\{x\in G\colon d_Z(x,P)<\log r\}\subset P_{r^2,B_N(n_Q,C^*\delta^k)}(a),
		\end{equation} 
		where $ {C^*}=C_1+\sqrt{2}$. In particular,  if $D(\mu_N,C^*/c)>0$ is as in~\eqref{doubling}, then
		\[ \mu\left(P^*\right)\leq C_4\mu\left(P\right),\qquad\text{where }\quad C_4=2D\Bigl(\mu_N,\frac{C^*}{c}\Bigr).\]
	\end{lem}
	\begin{proof}
		Let $x=(n,1)\exp(tZ)\in P^*$ and $y=(m,1)\exp(s Z)\in P$ such that $d_Z(x,y)<\log r$. Then,
		\begin{align}\label{stima}
			\cosh(\log r)&> \cosh\left(d_Z(x,y)\right)=\cosh\Bigl(d_G    \bigl((n,e^t),(m,e^s)\bigr)\Bigr)\nonumber\\  
			&=\cosh\left(|t-s|\right)+\frac{1}{2e^{t+s}}d_N(n,m)^2.
		\end{align}
		On the one hand, this gives $|t-s|<\log r$, which implies $t\in U_{r^2}(a)$. On the other hand, if $P$ is large admissible, from \eqref{stima} one gets
		\begin{equation*}    d_N\left(n,m\right)^2<2e^{t+s}\cosh(\log r)< 2a^2r^4\leq 2\delta^{2k},
		\end{equation*}
		while if $P$ is small admissible, again starting from \eqref{stima} but using a finer estimate we get
		\begin{align*}       d_N\left(n,m\right)^2&<2e^{t+s}\left(\cosh(\log r)-1\right)=4e^{t+s}\left(\sinh\left(\frac{\log r}{2}\right)\right)^2\\
			&< 4a^2r^3\log^2 r\leq\frac{4}{e}(ae^2\log r)^2\leq\frac{4}{e}\delta^{2k}.
		\end{align*}
		Hence, for any admissible $P$, $d_N\left(n,m\right)\leq \sqrt{2}\delta^k$. Since $m\in B_N(n_Q,C_1\delta^k)$, it follows that $n\in B_N(n_Q,(C_1+\sqrt{2} )\delta^k)$. This proves~\eqref{lem 7.1}. The inequality involving the measures now simply follows from Proposition~\ref{prop-cyl}~$(vi)$,~\eqref{doubling}, and Theorem~\ref{christ}~$(iii)$. Indeed
		\begin{align*}
			\mu(P^*)&\leq 4\log r \mu_N(B_N(n_Q,C^*\delta^k))\\
			&\leq 4\log r D\Bigl(\mu_N,\frac{C^*}{c}\Bigr) \mu_N(B_N(n_Q,c\delta^k))\\
			&\leq 
			4\log r D\Bigl(\mu_N,\frac{C^*}{c}\Bigr)\mu_N(Q)=2D\Bigl(\mu_N,\frac{C^*}{c}\Bigr)\mu(P).
		\end{align*}
	\end{proof}
	
	We are now ready to complete the proof of our main result.
	
	\begin{proof}[Proof of Theorem \ref{main}]
		Set $C_0=\max\{C_1,C_3,C_4\}$ and let $P_j=P_{r_j,Q_j}(a_j)$ be the sets in $\mathcal{D}^Z$ for which, according to Theorem \ref{CZdec}, items (a), (b) and (c) of Definition~\ref{def CZP} hold. By means of Lemma~\ref{lem 1}, we know that
		\begin{itemize}
			\item[(d)] $P_j\subset B_G^Z(x_j,C_0R_j)$,
		\end{itemize}
		with $x_j=(n_{Q_j},1)\exp((\log a)Z)$ and $R_j=\log r_j$. Now, let $P_j^*=\{x: \ d_Z(x,P_j)<R_j\}$. By Lemma~\ref{lem 2} and the construction of $P_j$'s in Theorem~\ref{CZdec}, we get
		\[\sum_j\mu(P_j^*)\leq C_4\sum_j\mu(P_j)\leq \frac{C_4}{\alpha}\sum_j\int_{P_j}|f|\de \mu\leq \frac{C_4}{\alpha}\|f\|_1, \]
		which is item $(e)$ in Definition~\ref{def CZP}.
	\end{proof}
	
	It is immediate to see that, for any $Z\in\mathcal{Z}$ and any $\mu\in\mathcal{F}_Z$, the space $(G,d_Z,\mu)$ with the family $\mathcal{D}^Z$ satisfies what in \cite[Definition 3.3]{MOV} is called condition (C) (take $\mathcal{R}'=\mathcal{R}=\mathcal{D}^Z$ in that definition). This implies that one can introduce a suitable Hardy space $H^1(\mu)$ and a corresponding space $BMO(\mu)$ as in \cite{vallarino_tesi}. By \cite[Theorem 1.2]{HS} see also \cite[Theorem 3.2]{MOV} and \cite[Theorem 3.10]{vallarino_tesi}),
	as an immediate consequence of Theorem \ref{main} we have the following result concerning the boundedness of a class of integral operators on $(G,d_Z,\mu)$.
	
	\begin{thm}\label{integral operators 1}
		Let $Z\in\mathcal{Z}$ and $\mu\in\mathcal{F}_Z$.
		Let $T=\sum_{j\in\mathbb Z} T_j$ be a linear operator bounded on $L^2(\mu)$, where the $T_j's$ are integral operators with kernel $K_j$ and the series converges in the strong operator topology on $L^2(\mu)$. Assume that there exist positive constants $b, B, \varepsilon$ and $C > 1$ such that
		\begin{equation*}
			\begin{split}
				&\int_G|K_j(x,y)| (1+C^j d_Z(x,y))^\varepsilon d\mu(x)\leq B, \quad  y\in G,\\
				&\int_G |K_j(x,y)-K_j(x,z)|d\mu(x)\leq B (C^j d_Z(y,z))^b \quad  y,z\in G.
			\end{split}
		\end{equation*}
		Then $T$ extends from $L^1(\mu)\cap L^2(\mu)$ to an operator of weak type $(1,1)$, bounded on $L^p(\mu)$, for $1 < p \leq 2$ and bounded from $H^1(\mu)$ to $L^1(\mu)$.
	\end{thm}
	
	By duality, one can also obtain boundedness on $L^p(\mu)$ for $2\leq p<\infty$ and from $L^\infty$ to $BMO(\mu)$ for a class of operators whose kernels satisfy a dual version of the integral Hörmander condition appearing in the above theorem.

	\section{The CZP and the \texorpdfstring{$d_G$ metric}{Carnot--Cara\-th\'eodory}}\label{sec: CZP for dg}
	This section is devoted to discuss the relationship between the metric $d_Z$ introduced in the previous section and the \CaC metric $d_G$. In the first subsection we will see that, in the general case of $N$ nonabelian, $d_G=d_Z$ if $Z=X_0$, and if $N$ is Abelian, $d_Z$ is equivalent to $d_G$ for all $Z\in\mathcal{Z}$.
	This implies that for $Z=X_0$ (and any $\mu\in \mathcal{F}_{X_0}$) the metric $d_Z$ can be always substituted by $d_G$ in Theorem \ref{main}. If $N$ is Abelian, such a substitution can be made for any $Z\in\mathcal{Z}$. One may wonder if, indeed, also when $N$ is nonabelian $d_Z$ can be substituted by $d_G$ for any $Z\in \mathcal{Z}$. In the second subsection we
	prove that the answer to the above question is negative, even for $\mu=\rho$. This, somehow, justifies our use of the new metric $d_Z$, which turns out to be more adapted to the context.
	
	\subsection{Comparison with known results}
	Let $Z= X_0$ be the vertical vector field. In this case, $n(t)=1_N$ for every $t\in\mathbb R$ and then $d_Z=d_G$. Then, as a special case of Theorem \ref{main}, we immediately have the following result.
	
	\begin{thm}\label{mov generalized}
		The measure metric space $(G,d_G,\mu)$ has the CZP with respect to the family $\mathcal{D}^{X_0}$ for every $\mu\in \mathcal{F}_{X_0}$.
	\end{thm}
	
	This result was previously known only for $\mu=\rho$ \cite[Theorem 3.20]{MOV}. Hence, Theorem \ref{mov generalized} can be considered, at the same time, a generalization and a new proof of the result by Martini, Ottazzi and Vallarino.
	
	The rest of the subsection is devoted to show that if $N$ is Abelian then the vector field $X_0$ in Theorem \ref{mov generalized} can be substituted by any $Z\in\mathcal{Z}$. 
	\begin{thm}\label{hs generalized}
		If $N=\mathbb R^m$, then the measure metric space $(G,d_G,\mu)$ has the CZP with respect to $\mathcal D^Z$ for every $Z\in\mathcal{Z}$ and every $\mu\in \mathcal{F}_Z$.
	\end{thm}
	
	This result was previously known only for $Z=X_0$ and for $\mu=\rho$ \cite[Lemma 5.1]{HS}. Hence, Theorem \ref{hs generalized} can be considered, at the same time, a generalization and a new proof of the result by Hebisch and Steger.
	
	We begin the discussion which will lead to the proof of Theorem \ref{hs generalized}. The semidirect group $G=\mathbb R^m\rtimes A$ is the affine group of $\mathbb R^{m+1}$ obtained from translations by vectors of $\mathbb R^m$ and by homogeneous dilations. The group $G$ can be realized as group of matrices in ${\rm GL}(m+1)$ as
	\[G=\Bigl\{g(n,a)=\begin{bmatrix}A&n\\0&1
	\end{bmatrix}\colon n\in\mathbb R^m,A=\diag(a,\dots,a)\in{\rm GL}(m),a\in A\Bigr\}, \]
	indeed the semidirect product is preserved by the matrix multiplication. The Lie algebra $\mathfrak g$ of $G$ is then
	\begin{equation*}
		\mathfrak g=\Bigl\{\begin{bmatrix}
			D&X\\
			0&0
		\end{bmatrix}\colon X\in\mathbb R^m,D=\diag(s,\dots,s)\in M(m),s\in\mathbb R\Bigr\}.
	\end{equation*}
	Furthermore, if we fix the canonical orthonormal basis $\breve{X}_1,\dots,\Breve{X}_m$ of $\mathfrak n_1=\mathfrak n\simeq\mathbb R^m$, then the \CaC metric on $N$ coincides with the Euclidean metric and the lifting $X_\ell$ of $\breve{X}_\ell$ is such that
	\[(X_\ell)_{ij}=\begin{dcases*}
		1&\text{if }$(i,j)=(\ell,m+1)$,\\
		0&\text{otherwise,}
	\end{dcases*}\]
	for every $1\leq \ell\leq m$. Let $d_G$ be the \CaC induced by the vector fields $X_0,X_1,\dots,X_m$, where $X_0|_{(n,a)}=X_0|_a$, $X_0=\partial_a\in\mathfrak a$.
	
	A vector field $Z\in\mathfrak g$ with $\langle Z,X_0\rangle=1$ has the form
	\begin{equation}\label{ZAbeliano}
		Z=X_0+\sum_{\ell=1}^m\beta_\ell X_\ell,\qquad\beta_1,\dots,\beta_m\in\mathbb R. 
	\end{equation}
	By direct computation, one can easily obtain that for every $t\in\mathbb R$
	\[\exp(tZ)=\begin{bmatrix}
		\diag(e^t,\dots,e^t)&(e^t-1)\beta\\0&0
	\end{bmatrix}, \]
	where $\beta^T:=(\beta_1,\dots,\beta_m)\in\mathbb R^m$. In particular, this means that 
	\begin{equation}\label{ntAbeliano}
		n(t)=(e^t-1)\beta.
	\end{equation}
	\begin{prop}\label{dg equiv dz}
		If $G=\R^m\rtimes A$ and $Z$ is as in~\eqref{ZAbeliano}, then the distance $d_Z$ is left-invariant and is equivalent to $d_G$.
	\end{prop}
	\begin{proof}
		Let $(n,a),(n',a')\in G$. By~\eqref{distZ} and~\eqref{ntAbeliano}, we have that
		\begin{align*}
			\cosh\Bigl(d_Z&\bigl((n',a')^{-1}(n,a),1_G\bigr)\Bigr)=\cosh\left(d_Z\Bigl(\Bigl(\frac{n-n'}{a'},\frac a{a'}\Bigr),1_G\Bigr)\right)\\
			&=\cosh\Bigl(\log\Bigl(\frac a{a'}\Bigr)\Bigr)+\frac{a'}{2a}\Bigl|\frac{n-n'}{a'}-\Bigl(\frac a{a'}-1\Bigr)\beta\Bigr|^2\\
			&=\cosh\Bigl(\log\Bigl(\frac a{a'}\Bigr)\Bigr)+\frac{1}{2aa'}|n-n'+(a'-a)\beta|^2\\
			&=\cosh\Bigl(\log\Bigl(\frac a{a'}\Bigr)\Bigr)+\frac{1}{2aa'}|(n-\beta(a-1))-(n'-\beta(a'-1))|^2\\
			&=\cosh\Bigl(d_Z\bigl((n,a),(n',a')\bigr)\Bigr),
		\end{align*}
		which proves the left-invariance of $d_Z$.
		
		Since both $d_G$ and $d_Z$ are left-invariant, to prove that they define equivalent metrics it is sufficient to check that there exists a constant $D\geq 1$ such that
		\begin{equation*}
			D^{-1} d_Z\bigl((n,a),1_G\bigr)\leq d_G\bigl((n,a),1_G\bigr) \leq D d_Z\bigl((n,a),1_G)\bigr),
		\end{equation*}
		for every $(n,a)\in G$.
		For simplicity, we apply the change of coordinate $t=\log a$. Observe that by~\eqref{ntAbeliano}
		\begin{equation*}
			\frac{|n(t)|^2}{e^t}=\frac{|\beta|^2(e^t-1)^2}{e^t}=2|\beta|^2(\cosh{t}-1).
		\end{equation*}

		Then, by the definition of $d_Z$ and equation~\eqref{CC distance},
		\begin{equation*}
			\begin{split}
				\cosh\Big(d_Z((n,e^t),1_G)\Big)&=\cosh\Big(d_G((n-n(t),e^t),1_G)\Big)=\cosh(t)+\frac{|n-n(t)|^2}{2e^t}\\
				&\leq \cosh(t)+\frac{|n|^2}{e^t}+\frac{|n(t)|^2}{e^t}\leq D\cosh\Big(d_G((n,e^t),1_G)\Big),
			\end{split}
		\end{equation*}
		where $D=\max\{2|\beta|^2+1,2\}$. On the other hand,
		\begin{equation*}
			\begin{split}
				\cosh\Big(d_G((n,e^t),1_G)\Big)
				&\leq \cosh(t)+\frac{|n-n(t)|^2}{e^t}+\frac{|n(t)|^2}{e^t}\\
				&\leq D\cosh\Big(d_Z((n,e^t),1_G)\Big).
			\end{split}
		\end{equation*}
		Since $D\geq 2$, we can choose a constant $C$ such that $D\cosh{x}\leq \cosh{(D x)}$ for $x\geq C$. But if $|t|\geq C$, by~\eqref{CC distance} we have that both $d_G((n,e^t),1_G)$ and $d_Z((n,e^t),1_G)$ are greater or equal to $C$, and therefore
		\begin{equation*}
			\frac{1}{D}d_G((n,e^t),1_G)\leq d_Z((n,e^t),1_G)\leq D d_G((n,e^t),1_G), \quad |t|\geq C. 
		\end{equation*}
		It is also clear that
		\begin{equation*}
			\lim_{|n|\to \infty}\frac{d_Z((n,e^t),1_G)}{d_G((n,e^t),1_G)}=1, \quad |t|<C.
		\end{equation*}
		It remains to study the behavior of the ratio of the distances when $(n,t)$ tends to $(0,0)$. In this case we have the asymptotic estimates
		\begin{align*}
			d_G((n,e^t),1_G)&=\arccosh\Bigl(\cosh(t)+\frac 1{2e^t}|n|^2\Bigr)\\
			&\sim \sqrt{2(\cosh(t)-1)+|n|^2} 
			\sim\sqrt{t^2+|n|^2},
		\end{align*}
		and
		\begin{align*}
			d_Z((n,e^t),1_G)&=\arccosh\Bigl(\cosh(t)+\frac 1{2e^t}|n+\beta(e^t-1)|^2\Bigr)\sim \sqrt{t^2+|n+\beta t|^2}.
		\end{align*}
		Then, for $(n,t)\to(0,0)$ we have
		\begin{equation}\label{asint in 0}
			\frac{d_G((n,e^t),1_G)^2}{d_Z((n,e^t),1_G)^2}\sim\frac{t^2+|n|^2}{t^2+|n+\beta t|^2}=\begin{cases}
				1\ &t=0,\\
				\displaystyle \Phi_{\bm\beta}\left(\frac n t\right) &t\neq 0,
			\end{cases}
		\end{equation}
		where $\Phi_{\bm\beta}\colon\mathbb R^m\to [0,+\infty)$ is defined by
		\[\Phi_{\beta}(v):=\frac{1+|v|^2}{1+|v+\beta|^2},\qquad v\in\mathbb R^m.\]
		Since $\Phi_{\bm\beta}$ is bounded from below and from above by uniform positive constants on $\mathbb R^m$, the same is true for the left hand side in~\eqref{asint in 0} for $(n,t)$ in a compact neighbor of $(0,0)$.

	\end{proof}
	
	\begin{proof}[Proof of Theorem \ref{hs generalized}]
		Since the CZP is invariant for equivalent metrics, by Theorem \ref{main} and Proposition \ref{dg equiv dz} we immediately get the result.  
	\end{proof}
	
	\begin{oss}
		A similar discussion to that carried out above also applies to Theorem \ref{integral operators 1}, which was previously known to hold only for $d_Z=d_G$ and $\mu_Z=\rho$, as consequences of \cite[Lemma 5.1]{HS} (when $N$ is Abelian) and \cite[Theorem 3.20]{MOV} (when $N$ is nonabelian). Although we do not provide explicit examples here, we can conclude that Theorem \ref{integral operators 1} provides boundedness results for a wider class of integral operators than those covered by \cite{HS} and \cite{MOV}.
	\end{oss}

	\subsection{Optimality of Theorem \ref{mov generalized}}\label{sec:counterexample}
	In this section we show that in Theorem \ref{mov generalized}, in general $X_0$ cannot be substituted by \textit{any} $Z\in \mathcal{Z}$, even for $\mu=\rho$. To construct a counterexample, we will consider the extended Heisenberg group, where some recurrent quantities that were abstract so far, such as $n(t)$ and  $d_N(n,1_N)$, can be made explicit.
	In the first subsection we prove some geometric lemmas, of possibly independent interest, holding true for vector fields $Z\in \mathcal{Z}_1$, the subclass of $\mathcal{Z}$ made of vector fields with nonzero components only in the first layer of $\mathfrak g$. In the second subsection we recall some general facts about the extended Heisenberg group $\mathbb H^1_e$ and compute the relevant quantities in this setting. Then, we construct the counterexample on $({\mathbb{H}_e^1},d_{\mathbb{H}_e^1},\rho)$, using a vector field  $Z\in\mathcal{Z}_1$.
	
	\subsubsection{Geometric lemmas for $\mathcal Z_1$}
	
	We introduce the subfamily of vector fields of $G$ having no components in higher layers of $\mathfrak n$, i.e.
	\[\mathcal Z_1:=\{Z\in\mathcal Z\colon Z\in \mathfrak n_1\oplus \mathfrak a\subset \mathfrak g\}.\]
	In other words, for every $Z\in\mathcal Z_1$, there exist $\beta_1,\dots,\beta_{q_1}\in\mathbb C$ such that 
	\[Z=X_0+\sum_{i=1}^{q_1}\beta_iX_{1,i}.\]
	
	\begin{oss}
		For every $Z\in\mathcal Z_1$ and $t\in \mathbb{R}$ it holds 
		\begin{equation}\label{dist dal trasportato in G}
			d_G(\exp(tZ),1_G)\leq |t|\Vert Z\Vert,
		\end{equation}
		where $\Vert Z\Vert=\langle Z,Z\rangle^{1/2}$, and the inner product is defined in Section \ref{sec: preliminaries}.
		
		Indeed, consider the curve $\gamma\colon[0,t]\to G$ defined by $\gamma(s):=\exp(sZ)$, joining $1_G$ and $\exp(tZ)$ in $G$. Now,
		\[\dot\gamma(s_0)=\frac{\de}{\de s}\Big|_{s=s_0}\exp(sZ)=Z|_{\exp(s_0Z)}\in HG,\]
		by Theorem~3.31 in~\cite{warner} and the fact that $Z\in\mathcal Z_1$. Hence, $\gamma$ is an horizontal curve of $G$. By the fact that $Z|_{\exp (sZ)}=(\de L_{\exp(sZ)})_{1_G}(Z)$ and the left-invariance of the metric tensor, we have that
		\[d_G(1_G,\exp(tZ))\leq\int_0^t\|\dot\gamma(s)\|\de s=\int_0^t\|Z\|\de s=|t|\|Z\|.\]
	\end{oss}
	\begin{lem}\label{lemma controesempio 1}
		Let $Z\in \mathcal Z_1$. For every $R>0$,
		\begin{equation*}
			P_{e^{R/(2\Vert Z \Vert)},B_N(1_N,R/2)}(1)\subset B_G(1_G,R).
		\end{equation*}
	\end{lem}
	\begin{proof}
		Let $\eta>0$ and $x=(n,1)\exp(tZ)$, with $n\in B_N(1_N,\eta)$, $t\in U_r(1)$, be an arbitrary point in $P_{r,B_N(1_N,\eta)}(1)$. By the triangular inequality,
		\begin{equation*}
			d_G(x,1_G)\leq d_G(x,(n,1))+d_G((n,1),1_G).
		\end{equation*}
		Now, by the left-invariance of $d_G$,~\eqref{dist dal trasportato in G}, and the fact that $|t|< \log r$ for $t\in U_r(1)$,
		\begin{equation*}
			d_G(x,(n,1))=d_G(\exp(tZ),1_G)\leq \Vert Z \Vert|t|<\Vert Z \Vert\log r,
		\end{equation*}
		while by \eqref{CC distance},
		\begin{equation*}
			\cosh\big(d_G((n,1),1_G)\big)=1+\frac{1}{2}d_N(n,1_N)^2\leq 1+\frac{\eta^2}{2}\leq\cosh\eta.
		\end{equation*}
		Gluing all together and choosing $r=e^{R/(2\Vert Z \Vert)}$ and $\eta=R/2$ we get the desired result.
	\end{proof}

	\begin{lem}\label{lemma controesempio 2}
		Let $Z\in \mathcal Z_1$ and $P=P_{r,Q}(a)$, where $Q\in \mathcal Q$ is a Christ cube in $N$. For every $R>0$,
		\begin{equation*}
			P_{re^{R/(2\Vert Z \Vert)}, n_Q \cdot \Psi_{r,a}(B_N(1_N,R/2))}(a)\subset \{x\in G\colon d_G(x,P)<R\},
		\end{equation*}
		where $\Psi_{r,a}$ is defined in Proposition~\ref{prop-cyl}~$(v)$.
	\end{lem}
	\begin{proof}
		Let $k\in\mathbb Z$ be the generation of $Q$. Recalling that, by $(iii)$ of Theorem~\ref{christ}, $n_Q B_N(1_N,c\delta^k)\subset Q$ and applying first  Lemma \ref{lemma controesempio 1} and then Proposition \ref{prop-cyl}~$(v)$, we get
		\begin{equation*}
			\begin{split}
				\{x\in G\colon d_G(x,P)<R\}&=\bigcup_{x\in P}B_G(x,R)=\bigcup_{x\in P}xB_G(1_G,R)=PB_G(1_G,R)\\
				&\supset   P_{r,n_Q B_N(1_N,c\delta^k)}P_{e^{R/(2\Vert Z \Vert)},B_N(1_N,R/2)}(1)\\
				&\supset P_{re^{R/(2\Vert Z \Vert)}, n_Q B_N(1_N,c\delta^k)\cdot \Psi_{r,a}(B_N(1_N,R/2))}(a)\\
				&\supset P_{re^{R/(2\Vert Z \Vert)}, n_Q \cdot \Psi_{r,a}(B_N(1_N,R/2))}(a),
			\end{split}
		\end{equation*}
		where the last inclusion simply follows from the fact that $1_N\in B_N(1_N,c\delta^k)$.  
	\end{proof}

	\subsubsection{The counterexample}
	
	We consider the Heisenberg group $\HH^1$, that is $\R^3$ endowed with the product
	\[(q,p,\tau)\cdot_{\HH^1}(q',p',\tau'):=\Bigl(q+q',p+p',\tau+\tau'-\frac{1}{2}(qp'-pq')\Bigr).\]
	The neutral element is then $(0,0,0)$ and $(q,p,\tau)^{-1}=(-q,-p,-\tau)$, $\HH^1$ is nilpotent, hence unimodular. A Haar measure is $\de q\de p\de\tau$ and we denote by $|\cdot|$ the Haar measure of sets. 
	The Heisenberg group admits the following realization inside $\Sptwor$, namely
	\[N:=\Bigl\{n(q,p,\tau):=\begin{bmatrix}
		1&0&0&0\\
		p&1&0&0\\
		\tau&-q/2&1&-p\\
		-q/2&0&0&1\\
	\end{bmatrix}\colon q,p,\tau\in\R \Bigr\}\subset \Sptwor. \]
	The group law inherited from the matrix multiplication coincides with the classical group law of the one dimensional Heisenberg group
	\begin{align*}
		n(q,p,\tau)n(q',p',\tau')&=n((q,p,\tau)\cdot_{\HH^1}(q',p',\tau')).
	\end{align*}
	The group $ A$ acts on $\HH^1$ via the dilations $D_a\colon \mathbb H^1\to\mathbb H^1$, defined for $a\in A$ by $D_a(q,p,\tau):=(aq,ap,a^2\tau)$. 
	In $\Sptwor$, such dilations coincide with the conjugation by the matrices $ A_a:=\diag(a^{-1},1,a,1)\in\Sptwor$, that is
	\[ A_an(q,p,\tau) A_a^{-1}=n(aq,ap,a^2\tau)=n(D_a(q,p,\tau)). \]
	The extended Heisenberg group is the semidirect product $\HH^1_e=\HH^1\rtimes A$, that is $\R^3\times  A$ endowed by the product
	\[(q,p,\tau;a)\cdot_{\HH^1_e}(q',p',\tau';a'):=((q,p,\tau)\cdot_{\HH^1}D_a(q',p',\tau');aa'),\]
	and it is realized in $\Sptwor$ by
	\begin{align*}
		G:&=\{{g(q,p,\tau;a):=}n(q,p,\tau) A_a^{-1}\colon q,p,\tau\in \R,a\in A\}\\
		&=\Bigl\{\begin{bmatrix}
			a^{-1}&0&0&0\\
			a^{-1}p&1&0&0\\
			a^{-1}\tau&-q/2 &a&-p\\
			-q/(2a)&0&0&1\\
		\end{bmatrix}\colon q,p,\tau\in\R,a\in A \Bigr\}\subset\Sptwor,
	\end{align*}
	endowed with the matrix multiplication. Indeed
	\begin{align*}
		g(q,p,\tau;a)g(q',p',\tau';a')&=n(q,p,\tau) A_a^{-1}n(q',p',\tau') A_a A_a^{-1} A_{a'}^{-1}\\
		&=n((q,p,\tau) \cdot_{\HH^1}D_a(q',p',\tau')) A_{aa'}^{-1}\\
		&=g((q,p,\tau;a)\cdot_{\HH^1_e}(q',p',\tau';a')).
	\end{align*}
	The group $G$ coincides with the classical shearlet group, see~\cite{DahDeMDeV} and~\cite{Shearlets}. 
	The Lie algebra of $\Sptwor$ is:
	\begin{align*}
		\mathfrak{sp}(2,\R)=&\{X\in\mathfrak{gl}(2,\R)\colon XJ+JX=0\},
	\end{align*}
	where $J$ is the 
	standard symplectic form
	\begin{equation*}
		J=\begin{bmatrix}
			0&I_2\\-I_2&0
		\end{bmatrix}.
	\end{equation*}

	The Lie algebra $\mathfrak{sp}(2,\mathbb R)$ is semisimple and has Cartan involution $\Theta X=-\phantom{}^tX$, relative to which it has the Cartan decomposition $\mathfrak{sp}(2,\mathbb R)=\mathfrak{k}+\mathfrak{p}$, where $\mathfrak{k}$ and $\mathfrak{p}$ are the $+1$ and $-1$ eigenspaces of $\Theta$, respectively. The standard maximal Abelian subspace of $\mathfrak p$ is
	\[\ga=\Bigl\{H_{a,b}:=\begin{bmatrix}
		-a&0&0&0\\0&-b&0&0\\0&0&a&0\\0&0&0&b
	\end{bmatrix}\colon a,b\in\R\Bigr\}.\]
	The two linear operators on $\ga$ given by $\alpha(H_{a,b})=a-b$ and $\beta(H_{a,b})=2b$ provide a natural basis of simple roots. In fact,
	\begin{equation*}
		X_\alpha:=\begin{bmatrix}
			0&0&0&0\\1&0&0&0\\0&0&0&-1\\0&0&0&0
		\end{bmatrix},\quad X_\beta:=\begin{bmatrix}
			0&0&0&0\\0&0&0&0\\0&0&0&0\\0&1/2&0&0
		\end{bmatrix},
	\end{equation*}
	\begin{equation*}
		X_{\alpha+\beta}:=\begin{bmatrix}
			0&0&0&0\\0&0&0&0\\0&-1/2&0&0\\-1/2&0&0&0
		\end{bmatrix},\quad X_{2\alpha+\beta}:=\begin{bmatrix}
			0&0&0&0\\0&0&0&0\\1&0&0&0\\0&0&0&0
		\end{bmatrix}
	\end{equation*}
	satisfy for every $H\in\ga$
	\begin{align*}
		[H,X_\alpha]&=\alpha(H)X_\alpha;\\
		[H,X_\beta]&=\beta(H)X_\beta;\\
		[H,X_{\alpha+\beta}]&=(\alpha+\beta)(H)X_{\alpha+\beta};\\
		[H,X_{2\alpha+\beta}]&=(2\alpha+\beta)(H)X_{2\alpha+\beta}.
	\end{align*}
	Futhermore, $[X_\alpha,X_\beta]=X_{\alpha+\beta}$ and $[X_\alpha,X_{\alpha+\beta}]=X_{2\alpha+\beta}$.
	
	The Heisenberg algebra is 
	\[ \mathfrak n:=\Span\{X_\alpha,X_{\alpha+\beta},X_{2\alpha+\beta}\}\subset \mathfrak{sp}(2,\R).\]
	The extended Heisenberg algebra is 
	\begin{align*}
		\mathfrak g:&=\Span\{X_\alpha,X_{\alpha+\beta},X_{2\alpha+\beta},H_{1,0}\}\subset \mathfrak{sp}(2,\R).
	\end{align*}
	The Lie algebra $\mathfrak n$ is stratified with 2 layers and it has homogeneous dimension 4. The first layer is generated by $X_\alpha$ and $X_{\alpha+\beta}$. The \CaC metric associated to $X_\alpha$ and $X_{\alpha+\beta}$ is the left-invariant metric induced by the Koranyi norm
	\begin{equation}\label{Koranyi}
		\|(q,p,\tau)\|_{\mathbb H^1}^4=\frac 1{16}(q^2+p^2)^2+\tau^2,\qquad(q,p,\tau)\in \mathbb{H}^1,
	\end{equation}
	i.e.~$d_N((q,p,\tau),1_{\mathbb H^1})=\|(q,p,\tau)\|_{\mathbb H^1}$. We refer to~\cite{CapDaiPauTys} for more details on this distance.
	
	From now on, we shall adopt a little abuse of notation, by renaming $N$ and $G$ in $\Sptwor$ with $\mathbb H^1$ and $\mathbb H^1_e$, respectively, and by putting $(q,p,\tau)\in \mathbb H^1$ and $(q,p,\tau,a)\in \mathbb H^1_e$.
	
	The following geometric Lemma will be useful in the construction of our counterexample.

	\begin{lem}\label{coniugazione in Heisenberg}
		For any $R>0$ and $t\in \R$,
		\[B_{\mathbb{H}^1}(1_{\mathbb{H}^1},\tilde c e^tR)\subset \psi_t(B_{\mathbb{H}^1}(1_{\mathbb{H}^1},R)),\]
		where $\psi_t$ is defined in Proposition~\ref{prop-cyl} and $\tilde c^4=1/20$.
	\end{lem}
	\begin{proof}
		For every $L,M>0$ we put 
		\[Q(L,M):=[-L,L]^2\times[-M,M]\subset {\mathbb{H}^1},\]
		and $Q(L):=Q(L,L^2)$.
		It is easy to see that
		\begin{equation}\label{cubi in palle}
			Q(2\tilde c L)\subset B_{\mathbb{H}^1}(1_{\mathbb{H}^1},L)\subset Q(2L).
		\end{equation}
		Furthermore, for every $t\in\R$
		, by~\eqref{n(t)heis}, we have that if $(q,p,\tau)\in {\mathbb{H}^1}$, then
		\[\psi_t(q,p,\tau)=n(t)D_{e^t}(q,p,\tau)n(t)^{-1}=(e^tq,e^tp,e^{2t}\tau+e^t(e^t-1)q),\]
		from which it follows for every $L>0$
		\begin{equation}\label{psi cubo}
			\psi_t(Q(L))=Q(e^tL,e^{2t}L^2+e^t|1-e^t|L)\supset Q(e^tL).
		\end{equation}
		Therefore, by~\eqref{cubi in palle} and~\eqref{psi cubo},
		\begin{equation*}    \psi_t\bigl(B_{\mathbb{H}^1}(1_{\mathbb{H}^1},R)\bigr)\supset \psi_t\bigl(Q(2\tilde c R)\bigr)\supset Q(\tilde 2c e^tR)\supset B_{\mathbb{H}^1}(1_{\mathbb{H}^1},\tilde c e^t R),
		\end{equation*}
		as desired.
	\end{proof}
	We are now ready to show that for a suitable vector field $Z\in\mathcal{Z}_1$ the family $\mathcal D^Z$ is not a CZ family for the extended Heisenberg group equipped with a right Haar measure and the \CaC distance $d_{\mathbb{H}_e^1}$. We consider the vector field 
	\begin{equation}\label{Z Heisenberg}
		Z=X_\alpha+H_{1,0}=\begin{bmatrix}
			-1&0&0&0\\1&0&0&0\\0& 0&1&-1\\0&0&0&0
		\end{bmatrix}\in\mathfrak g. 
	\end{equation}
	
	Clearly $\langle Z,H_{1,0}\rangle=1\neq 0$, so that $Z\in \mathcal Z_1$. Furthermore, $\|Z\|=\sqrt{2}$
	.
	Let $t\in\R$. By explicit computation, we have that 
	\begin{align*}
		\exp(tZ)
		&=\begin{bmatrix}
			e^{-t}&0&0&0\\
			1-e^{-t}&1&0&0\\
			0&0&e^{t}&1-e^{t}\\
			0&0&0&1
		\end{bmatrix}=(0,e^t-1,0,e^t)\in {\mathbb{H}_e^1},
	\end{align*}
	that means
	\begin{equation}\label{n(t)heis}
		n(t)=(0,e^t-1,0)\in {\mathbb{H}^1}.
	\end{equation}
	We recall that $\rho$ is a right Haar measure of $\mathbb H^1_e$, which is given by $\de\rho=a^{-1}\de q \de p\de \tau \de a$. 
	
	\begin{thm}\label{thm: counterexample}
		Let $Z$ be the vector field given by \eqref{Z Heisenberg}. Then $\mathcal D^Z$ is not a CZ family for $({\mathbb{H}_e^1},d_{\mathbb{H}_e^1},\rho)$.  
	\end{thm}
	
	\begin{proof}
		It is enough to exhibit a sequence of admissible cylinders $\{P^\ell\}$ such that
		\begin{equation}\label{ratio measures counterexample}
			\sup_{\ell}\frac{\rho\left(\{x\colon d_{\mathbb{H}_e^1}(x,P^\ell)<K\diam P^\ell\}\right)}{\rho(P^\ell)}=+\infty, \quad \text{for every } K>0. 
		\end{equation}
		Indeed, if~$(d)$ in Definition~\ref{def CZP} holds, namely, if there exists $C>0$ such that for every $P^\ell\in\mathcal D^Z$ there is $R_\ell>0$ such that $\diam P^\ell\leq CR_\ell$, then by \eqref{ratio measures counterexample} one would get
		\begin{equation*}
			\sup_\ell\frac{\rho\left(\{x\colon d_{\mathbb{H}_e^1}(x,P^\ell)<R_\ell\}\right)}{\rho(P^\ell)}\geq \sup_\ell\frac{\rho\left(\{x\colon d_{\mathbb{H}_e^1}(x,P^\ell)<\frac 1C\diam P^\ell\}\right)}{\rho(P^\ell)}=+\infty,
		\end{equation*}
		which is a contradiction of property~$(e)$ in Definition~\ref{def CZP}. Hence, $(d)$ and $(e)$ cannot hold together and $\mathcal D^Z$ is not a CZ family for $({\mathbb{H}_e^1},d_{\mathbb{H}_e^1},\rho)$.
		
		We consider the family of dyadic sets $\mathcal D^Z$ build as in the proof of Theorem~\ref{thm dyadic partitions} starting from the family $\{P_k\}_{k\in\mathbb N}$ with $P_0=P_{{r_0},Q_0}(1)$ for some $r_0>e$ and some Christ dyadic cube $Q_0$ in $\mathbb H^1$. We consider the subsequence of \{$P_k$\} such that, for $\ell\geq 0$, $P_{k_\ell+1}= p^\downarrow(P_{k_\ell})$. By simplicity, we denote by $r\colon\mathcal D^Z\to (1,+\infty)$ the map defined by $r(P_{r',Q'}(a')):=r'$. An immediate computation shows that 
		\[\log_{r_0} r(P_{k_\ell+1}\setminus P_{k_\ell})=4\cdot 6^\ell,\qquad \ell\in\mathbb N.\]
		
		Now observe that for every $P'\in s(P)$, $\log_{r_0} r(P')$ is equal to $1/2\log_{r_0} r(P)$ if $s(P)$ is as~$(i)$ in Definition~\ref{def sons}, or to $\log_{r_0} r(P)$ if $s(P)$ is as~$(ii)$ in Definition~\ref{def sons}. 
		We denote by $m(\ell)\in\mathbb N$ the smallest number of iterations of the set-valued function $s$ on $P_{k_\ell+1}\setminus P_{k_\ell}$ in which the case  $(i)$ in Definition~\ref{def sons} occurs exactly $\lfloor\ell \log_2 3\rfloor+\ell+2$ times. Then $\log_{r_0} r(P')\in [1,2]$ for every $P'\in s^{m(\ell)}(P_{k_\ell+1}\setminus P_{k_\ell})$. Indeed,
		\[\log_{r_0}r(P')=\frac{\log_{r_0} r(P)}{2^{\lfloor\ell \log_2 3\rfloor+\ell+2}}=\frac{4\cdot 6^\ell}{4\cdot 2^\ell\cdot 2^{\lfloor\ell \log_2 3\rfloor}}=\frac{3^\ell}{2^{\lfloor\ell \log_2 3\rfloor}}, \]
		and then 
		\[1=\Bigl(\frac{3}{2^{\log_2 3}}\Bigr)^\ell\leq \log_{r_0}r(P')=\frac{3^\ell}{2^{\lfloor\ell \log_2 3\rfloor}}\leq \frac{3^\ell}{2^{\ell\log_2 3-1}}=2. \]
		Furthermore, by the construction of $\mathcal D^Z$, we have that $s^{m(\ell)}(P_{k_\ell+1}\setminus P_{k_\ell})\subset \mathcal D^Z$.

		Again by the construction of the $P_k$'s, one can see that $(1_{\mathbb H^1},{r_0}^{-4\cdot 6^{\ell}})\in P_{k_\ell+1}\setminus P_{k_\ell}$. Then, fix a $P^\ell\in {s}^{m(\ell)}(P_{k_\ell+1}\setminus P_{k_\ell})\subset\mathcal D^Z$ such that $(1_{\mathbb H^1},{r_0}^{-4\cdot 6^{\ell}})\in\overline{P^\ell}$. We put $P^\ell=P_{r(P^\ell),Q_\ell}(a_\ell)$ and $k(\ell)\in \mathbb Z$ such that $Q_\ell\in\mathcal Q_{k(\ell)}$. 
		Clearly, $r(P^\ell)\in [{r_0},{r_0}^2]$. It is immediate from the definition of admissible cylinder that
		\begin{equation}\label{condizaj}
			a_\ell\in \Bigl[\frac{1}{{r_0}^{4\cdot 6^\ell}r(P^\ell)},\frac{r(P^\ell)}{{r_0}^{4\cdot 6^\ell}}\Bigr]\subset [{r_0}^{-4\cdot 6^\ell-2},{r_0}^{-4\cdot 6^\ell+2}],
		\end{equation}
		\begin{equation}\label{condizdeltaj}
			\delta^{k(\ell)}\in[a_\ell r(P^\ell)^2,\lambda a_\ell r(P^\ell)^\gamma]\subset [a_\ell {r_0}^2,\lambda a_\ell {r_0}^{2\gamma}].
		\end{equation}


		Then by  Lemma \ref{coniugazione in Heisenberg}, and observing that $U_{r(P^\ell)}(a_\ell)=(\log(a_\ell/r(P^\ell)),\log(a_\ell r(P^\ell)))$, we have that for every $R>0$
		\begin{equation*}
			\begin{split}
				\Psi_{r(P^\ell),a_\ell}\Bigl(B_{\mathbb{H}^1}\Bigl(1_{\mathbb{H}^1},\frac R2\Bigr)\Bigr)&=\bigcap_{t\in U_{r(P^\ell)}(a_\ell)}\psi_t\Bigl(B_{\mathbb{H}^1}\Bigl(1_{\mathbb{H}^1},\frac{R}{2}\Bigr)\Bigr)\\
				&\supset \bigcap_{t\in U_{r(P^\ell)}(a_\ell)} B_{\mathbb{H}^1}\Bigl(1_{\mathbb{H}^1},\frac{\tilde c}2 e^tR\Bigr)=B_{\mathbb{H}^1}\Bigl(1_{\mathbb{H}^1},C'a_{\ell}R\Bigr),
			\end{split}
		\end{equation*}
		where $C'=\tilde c/2{r_0}^2$, because $r(P^\ell)\leq {r_0}^2$.
		Hence, applying Lemma \ref{lemma controesempio 2},  recalling that $\Vert Z\Vert=\sqrt{2}$ and  that $r(P^\ell)\geq {r_0}$, it follows that for any $R>0$
		\begin{equation*}
			P_{e^{R/(2\sqrt{2})}{r_0}, n_{Q_\ell} \cdot B_{\mathbb{H}^1}(1_{\mathbb{H}^1},C'a_\ell R)}(a_\ell)\subset \{x\in\mathbb H^1_e \colon d_{\mathbb{H}_e^1}(x,P^\ell)<R\}.
		\end{equation*}
		Then, by means of Proposition~\ref{prop-cyl}~$(vi)$, we get
		\begin{equation*}
			\begin{split}
				\rho(\{x\colon d_{\mathbb{H}_e^1}(x,P^\ell)<R\})\geq 2{\Bigl(\frac{R}{2\sqrt 2}+\log {r_0}\Bigr)} |B_{\mathbb{H}^1}(1_{\mathbb{H}^1},C'a_\ell R)|\approx a_\ell^4 R^5,
			\end{split}
		\end{equation*}
		if $R\gg 1$.
		On the other hand, 
		\[\rho(P^\ell)=2\log r(P^\ell)|Q_\ell|\approx |B_{\mathbb{H}^1}(n_{Q_\ell}, c\delta^{k(\ell)})|\approx \delta^{4k(\ell)}\approx a_\ell^4,\]
		by~\eqref{condizdeltaj}, from which it follows that
		\begin{equation}\label{eq: ratio measures sequence}
			\frac{\rho\left(\{x\colon d_{\mathbb{H}_e^1}(x,P^\ell)<K\diam P^\ell\}\right)}{\rho(P^\ell)}\gtrsim (K\diam P^\ell)^5.
		\end{equation}
		
		Now we estimate the diameter of $P^\ell$ from below.
		Let 
		\[n_\ell:=n_{Q_\ell}\cdot(c\delta^{k(\ell)},0,0)\in n_{Q_\ell}\cdot B_{\mathbb{H}^1}(1_{\mathbb{H}^1},c\delta^{k(\ell)})\subset Q_\ell.\]
		By~\eqref{CC distance},~\eqref{n(t)heis}, and~\eqref{Koranyi}, we have    \begin{align*}
			\cosh(\diam P^\ell)&\geq \cosh\bigl(d_{\mathbb{H}_e^1}\bigl(n_\ell n(\log a_\ell),a_\ell),(n_{Q_\ell} n(\log a_\ell),a_\ell)\bigr)\bigr)\\
			&\geq\frac 1{2a_\ell^2} \|n(\log a_\ell)^{-1}n_\ell^{-1}n_{Q_\ell}n(\log a_\ell)\|_{\mathbb{H}^1}^2\\
			&=\frac 1{2a_\ell^2}\|(c\delta^{k(\ell)},0,c\delta^{k(\ell)} (1-a_\ell))\|_{\mathbb{H}^1}^2\\
			&=\frac 1{2a_\ell^2} \sqrt{\frac{c^4\delta^{4k(\ell)}}{16}+c^2\delta^{2k(\ell)}(1-a_\ell)^2}\\
			&\gtrsim \sqrt{\left(\frac{\delta^{k(\ell)}}{a_\ell}\right)^4+\left(\frac{\delta^{k(\ell)}}{a_\ell}\right)^2\left(\frac{1-a_\ell}{a_\ell}\right)^2}\\
			&\gtrsim \frac{1-a_\ell}{a_\ell}\gtrsim \frac{1}{a_\ell}\gtrsim {r_0}^{4\cdot 6^\ell},
		\end{align*}
		since $\delta^{k(\ell)}/a_\ell\geq {r_0}^2$ by~\eqref{condizdeltaj}, and $a_\ell\lesssim {r_0}^{-4\cdot 6^\ell}$ by~\eqref{condizaj}.
		This implies that $\diam P^\ell\gtrsim 6^\ell$. Then, by \eqref{eq: ratio measures sequence} we get that for any $K>0$,
		\begin{equation*}
			\frac{\rho\left(\{x\colon d_{\mathbb{H}_e^1}(x,P^\ell)<K\diam P^\ell\}\right)}{\rho(P^\ell)}\gtrsim K^5 6^{5\ell}\longrightarrow\infty, \quad \ell\to\infty.
		\end{equation*}
	\end{proof}
	
	\bibliographystyle{abbrv}
	\end{document}